\documentclass[12pt]{amsart}
\usepackage[T1]{fontenc}
\usepackage{latexsym} 
\usepackage{amsfonts, amsthm, amsmath,amssymb}
\usepackage{BOONDOX-calo}
\usepackage{geometry}
\usepackage{times}
\usepackage{xcolor}
\definecolor{refkey}{rgb}{0,0,1}
\definecolor{labelkey}{rgb}{1,0,0}
\usepackage{graphicx}
\usepackage{float}
\usepackage{mdframed}
\usepackage{ulem}
\definecolor{darkblue}{rgb}{0.0, 0.0, 0.55}
\definecolor{darkcerulean}{rgb}{0.03, 0.27, 0.49}
\definecolor{darkpowderblue}{rgb}{0.0, 0.2, 0.6}
\definecolor{britishracinggreen}{rgb}{0.0, 0.26, 0.15}
\newenvironment{blu}{\color{darkpowderblue}}{}
\newenvironment{mgg}{\color{magenta}}{}
\newenvironment{brg}{\color{britishracinggreen}}{}
\newenvironment{red}{\color{red}}{}
\newcommand{\bre}{\begin{red}}
	\newcommand{\ere}{\end{red}}
\newcommand{\bas}{\begin{brg}}
	\newcommand{\eas}{\end{brg}}
\newcommand{\bblu}{\begin{blu}}
	\newcommand{\eblu}{\end{blu}}
\newcommand{\bmag}{\begin{mgg}}
	\newcommand{\emag}{\end{mgg}}

\newcommand{\xch}[2]{#1}
\newcommand{\reftext}[1]{#1}

\NewDocumentCommand{\colorrule}{O{.4pt}m}{{\color{#2}\hrule height#1}\vspace{4mm}}
\newtheorem{thm}{Theorem}[section]
\newtheorem{prop}[thm]{Proposition}
\newtheorem{conj}[thm]{Conjecture}
\newtheorem{lem}[thm]{Lemma}
\newtheorem{coro}[thm]{Corollary}
\theoremstyle{definition}
\newtheorem{defn}[thm]{Definition}
\newtheorem{rem}[thm]{Remark}
\newtheorem{Example}[thm]{Example}
\numberwithin{equation}{section}

\def\beq{\begin{equation}} 
	\def\eeq{\end{equation}}

\title{Spectral Metric and Einstein Functionals} 
\author[L.\ D\k{a}browski]{Ludwik D\k{a}browski${}^{(1)}$}
\address{${}^{(1)}$ SISSA (Scuola Internazionale Superiore di Studi Avanzati), \newline\indent Via Bonomea 265, 34136 Trieste, Italy} 
\email{dabrow@sissa.it} 
\author[A.\ Sitarz]{Andrzej Sitarz${}^{(2)}$}
\author[P.\ Zalecki]{Pawe\l{} Zalecki${}^{(2)}$}
\thanks{This work is supported by the Polish National Science Centre grant 2020/37/B/ST1/01540}
\address{${}^{(2)}$ Institute of Theoretical Physics, Jagiellonian University, \newline\indent
	prof.\ Stanis\l awa \L ojasiewicza 11, 30-348 Krak\'ow, Poland.}
\email{andrzej.sitarz@uj.edu.pl}  
\email{pawel.zalecki@doctoral.uj.edu.pl}
\date{}
\begin{document}
	\maketitle
	\begin{abstract}
		We define bilinear functionals of vector fields and differential forms, the densities of
		which yield the metric and Einstein tensors on even-dimensional Riemannian manifolds.
		We generalise these concepts in non-commutative geometry and, in particular, we prove that for 
		the conformally rescaled geometry of the noncommutative two-torus the Einstein functional 
		vanishes.  
	\end{abstract}
	\section{Introduction}
	\label{sec1}
	
	Riemannian geometry flourishing for more than hundred and fifty years is
	continually at the core of modern mathematics with a wealth of applications
	and open problems. Its role as a key to understanding general relativity
	makes it ubiquitous in modern physics. The study of interconnection between
	the geometric objects on the manifolds and differential operators led to
	the development of spectral geometry and, with the seminal question \textit{``Can
		one hear the shape of a drum?''}, was popularised by Mark Kac
	\cite{Ka66}. A further generalization was brought by noncommutative geometry
	\cite{Co94}, where spectral methods were proposed to study invariants of
	algebras, functionals on noncommutative algebras and geometric objects
	that extends the notion of curvature, leading, for example, to the version
	of Gauss-Bonnet theorem for noncommutative tori.
	
	The aim of this paper is to propose a few new functionals that express
	other basic geometric objects, in particular the Einstein tensor, using
	the methods of the Wodzicki residue and its generalisation in noncommutative
	geometry.
	
	\subsection{An overview of the field}
	\label{sec1.1}
	
	An eminent spectral scheme that generates geometric objects on manifolds
	such as volume, scalar curvature, and other scalar combinations of curvature
	tensors and their derivatives \textit{prima facie} is the small-time asymptotic
	expansion of the (localised) trace of heat kernel \cite{Gi84,Gi04}. This
	has many diverse applications in mathematics and in physics, mostly in
	the context of general relativity and its generalisations.
	
	Using the Mellin transform, the coefficients of this expansion can be transmuted
	into certain values or residues of the (localised) zeta function of the
	Laplacian. In turn, they can be expressed using the Wodzicki residue
	$\mathcal{W}$ (also known as noncommutative residue), which is a unique
	(up to multiplication by a constant) tracial state on the algebra of pseudo-differential
	operators ($\Psi $DO) on a complex vector bundle $E$ over a compact manifold
	$M$ of dimension $n\geq 2$ \cite{Gu85,Wo87}. For the oriented manifold
	$M$ it is given up to multiplicative constant by a simple integral formula,
	%
	\begin{equation}
		\mathcal{W}\,(P) := \int _{M} \left ( \int _{|\xi |=1} tr\, \sigma _{-n}(P)(x,
		\xi )~ {\mathcal V}_{\xi}\right )~ d^{n} x,
		\label{eq1.1}
	\end{equation}
	where $tr$ is the trace over endomorphisms of the bundle $E$ at any given
	point of $M$, $\sigma _{-n}(P)$ is the symbol of order $-n$ of a pseudodifferential
	operator $P$ and ${\mathcal V}_{\xi}$ denotes the volume form on the unit
	sphere. The name `residue' comes from the fact that it is indeed a residue
	(at $z\!=\!1$) of the $\zeta $-function associated to $P$
	\cite{G-BVF01}. In this paper, for simplicity, we focus only on closed
	oriented $M$ of even dimension $n=2m$. In this case, for a Riemannian manifold
	$M$ equipped with a metric tensor $g$ and the (scalar) Laplacian
	$\Delta $ one has,
	%
	\begin{equation}
		\label{WresL}
		\mathcal{W}\,( \Delta ^{-m})= v_{n-1} \, vol(M),
	\end{equation}
	and a \textit{localized} form, a functional of $f \in C^{\infty}(M)$,
	%
	\begin{equation}
		\label{WresfL}
		{\mathcal{v}}(f) := \mathcal{W}\,(f \Delta ^{-m})= v_{n-1} \int _{M} f~vol_{g},
	\end{equation}
	where
	\begin{equation*}
		v_{n-1}:=vol(S^{n-1}) = \frac{2\pi ^{m}}{\Gamma (m)},
	\end{equation*}
	is the volume of the unit sphere $S^{n-1}$ in $\mathbb{R}^{n}$.
	
	A startling result regarding a higher power of the Laplacian was divulged
	by Connes in the early 1990s \cite{Co96} and explicitly confirmed independently
	in \cite{Ka95} and in \cite{KaWa95}. Namely, for $n>2$,
	%
	\begin{equation}
		\label{WresLL}
		\mathcal{W}\,( \Delta ^{-m+1})= \frac{n-2}{12} v_{n-1} \int _{M} R(g)~vol_{g},
	\end{equation}
	is up to a constant a Riemannian analogue of the Einstein-Hilbert action
	functional of general relativity in \xch{vacuum}{vaccum}. Here $R=R(g)$ is the scalar
	curvature, that is the $g$-trace $R\!=\!g^{jk}{Ric}_{jk}$ of the Ricci tensor
	with components ${Ric}_{jk}$ in local coordinates, where $g^{jk}$ are the raised components of the metric $g$.
	
	A localised form of \reftext{\eqref{WresLL}} for $n>2$ is a functional on
	$C^{\infty}(M)$ (\xch{cf.}{c.f.} \cite{CM08})
	%
	\begin{equation}
		\label{WresfLL}
		{\mathcal R}(f) := \mathcal{W}\,(f \Delta ^{-m+1})= \frac{n-2}{12} v_{n-1}
		\int _{M} f R(g) vol_{g} .
	\end{equation}
	For the Riemannian spin manifolds this can be expressed using such operators
	in spinor bundle as the spin Laplacian or the Dirac operator which are
	related by the Schr\"{o}dinger-Lichnerowicz formula. In particular, for the
	Dirac operator one has:
	%
	\begin{equation}
		\label{WresfD}
		\begin{aligned}
			\mathcal{W}\,(f|D|^{-n})= 2^{m} v_{n-1} \int _{M} f~vol_{g},
			\\
			\mathcal{W}\,(f |D|^{-n+2})= -2^{m} \frac{n-2}{24} v_{n-1} \int _{M}
			f R(g)~vol_{g} .
		\end{aligned}
		%
	\end{equation}
	
	In the noncommutative realm the spectral-theoretic approach to scalar curvature
	has been extended also to quantum tori in the seminal work of Connes and
	Tretkoff \cite{CoTr11}, expanded in \cite{CoMo14} and then extensively studied by
	many authors
	\cite{FK12,FK13,FK15,Fa15,FGK16,LM16,IM18,Li18,SZ19,Li20,Po20b}, see also
	\cite{Co19,FK20,LM20} for recent surveys. Therein, the pseudodifferential
	operators and symbol calculus introduced in \cite{Co80} and extended to
	crossed product algebras in \cite{Ba88a,Ba88b}, cf. also
	\cite{Ta18,HaLePo19a,HaLePo19b,Go-PeJuPa17} for detailed account, have
	been employed for computations of certain values and residues of
	$\zeta $-functions of suitable Laplace type operators.
	
	In these papers the analogue of conformal transformations of the standard
	(flat) geometry of the noncommutative 2-torus was worked out, with the
	conformal scaling (Weyl factor) taken as an invertible positive element
	from the underlying algebra $A$; see also \cite{CF19} for a conformally
	non-flat quantum 4-torus. These achievements produced extremely interesting
	modular aspects, which, however, involve complicated, often computer-assisted,
	calculations with thousands of terms.
	
	In this respect, the alternative use of Weyl factor from the copy
	$A^{o}$ of $A$ in the commutant of $A$ in \cite{DaSi13} and the early nonconformally
	flat modification of the torus geometry by an element from $A^{o}$ in
	\cite{DaSi15} allowed stripping the modular aspects and somewhat simplifying
	the computations, while still producing a nontrivial scalar curvature.
	One can also work with $A^{o}$ as the underlying algebra and modify the
	Dirac operator by an element from $A$ instead. This still yields the usual
	spectral triple, while taking the other two combinations yields the twisted
	(modular) triples, so there are just two different types of situation
	\cite{BDS19}.
	
	Now, for quantum tori, and more generally for spectral triples with real
	structure $J$, there is a natural choice for $A^{o}$ as $JAJ$. Moreover,
	since $A^{o}$ and $A$ coincide in the classical (commutative) case, also
	modifying the quantum torus geometry by elements from $A$ or from
	$A^{o}$ appear a viable option and both of them should be considered.
	The latter leads to the usual (untwisted) spectral triple with the algebra
	$A$.
	
	Concerning the spectral approach described above, the Wodzicki residue
	methods have already been used for quantum tori in dimensions n = 2 and
	n = 4 in \cite{FW11} and \cite{FK15}, respectively, and have been extended
	to any dimension $n\geq 2$ in \cite{LNP16}, see also \cite{Po20a}, and
	applied in \cite{Fa15,Po20b} for the scalar curvature on NC tori.
	
	Altogether, the aforementioned spectral methods allowed so far to extract
	the scalar invariants built from the metric and the curvature tensors,
	while the understanding of the curvature tensor itself has not yet been
	achieved. Only recently has the Ricci curvature tensor $\xch{Ric}{R}_{jk}$ been recovered
	using the Hodge-de\,Rham spectral triple, and more precisely from the difference
	of zeta functions of the Laplacian on functions and on one-forms, and applied
	to the noncommutative two-torus \cite{FGK16}.
	
	No doubt it would be extremely interesting to recover other important tensors
	in both the classical setup as well as for the generalised or quantum geometries.
	In this paper we accomplish this task for the metric tensor ${ g}$ itself,
	its dual, and for the Einstein tensor
	\begin{equation*}
		{ G}:= \text{Ric} - \frac{1}{2} R(g)\, {g},
	\end{equation*}
	which directly enters the Einstein field equations with matter, and its
	dual. For this purpose, we employ in fact the Wodzicki residue of a suitable
	power of the Laplace type operator multiplied by a pair of other differential
	operators. Notably, we demonstrate that the Wodzicki residue density recovers
	the tensors ${g}$ and ${G}$ as certain bilinear functionals of vector fields
	on a manifold $M$, while their dual tensors are recovered as the density
	of bilinear functionals of differential one-forms on $M$.
	
	Using Connes' pseudodifferential calculus on noncommutative tori we also
	propose a conspicuous quantum analogue of these functionals and probe it
	on 2 and 4-dimensional noncommutative tori. Note, however, that a direct
	comparison with \cite{FGK16} is arduous not only due to the difference
	between the canonical and Hodge-de\,Rham spectral triples, but more fundamentally,
	due to the unknown way how to extract from the Einstein tensor the Ricci
	tensor due to the ordering ambiguity (\textit{quantum product}) of $g$ and
	$R(g)$.
	
	The aforementioned functionals are built using either the Laplace or Dirac
	operator (in case $M$ is spin manifold), but before embarking on the latter
	one, we need actually first to settle other Laplace-type operators, acting
	on vector bundles of suitable rank over $M$. In fact, it will be advantageous
	to consider a general class of such operators and to control which tensors
	they produce. For this reason, our notation will systematically keep track
	of the dependence on a particular differential operator.
	
	This will be also beneficial for another reason. Namely, one typically
	obtains spectrally the scalar curvature $R$ of a manifold using the Laplace
	or Dirac operators, which are built openly from the Levi-Civita connection
	(so torsion-free). Classically, these operators can be also characterised
	by the fact that the so obtained $R$ minimises the relevant functionals
	(as otherwise there is a non-negative contribution from the torsion). This
	property may be not satisfied in the noncommutative realm, as actually
	there is no implicit notion of torsion (see, however, other frameworks
	in e.g. \cite{fgr99,Ro13,BM20,BGJ21a,BGJ21b}).
	
	It turns out that already on the noncommutative 4-torus the most immediate
	quantum analogue of the conformally rescaled \textit{flat} Laplace operator
	does not minimise the scalar curvature functional \cite{Si14}. Thus, if
	no particular reference Laplace-type operator is declared to be torsion-free,
	it can be convenient to label also the spectrally obtained geometric tensors
	such as $R$, ${\text{Ric}}$ or ${ G}$ by the concrete operator
	$\Delta $ which has been employed for their definition, and call them,
	e.g. ``$\Delta $ scalar curvature'', etc. By doing so, we will be able to
	determine which Laplace-type operators provide the same geometric invariant.
	
	\subsection{Organization of the article}
	\label{sec1.2}
	
	We start with a brief overview of the normal coordinates and the expression
	of the Laplace operator and its symbol at a given point on the manifold
	in these coordinates and present the explicit results for the symbols of
	the negative powers of the Laplace operator. Then we prove the main theorems
	showing that functionals on vector fields (understood as differential operators) yield the metric (\reftext{Theorem~\ref{metricthm}}) and the Einstein tensor (\reftext{Theorem~\ref{einsteinthm}}) densities.
	
	In Section~\ref{sec3} we demonstrate in \reftext{Theorems~\ref{thm3.1} and \ref{lifteinstein}} the value of the metric
	and Einstein functionals for Laplace-type operators on vector bundles,
	giving interesting examples of the spinor Laplacian (\reftext{Proposition \ref{prop3.4}}) and
	square of the Dirac operator (\reftext{Proposition \ref{prop3.5}}). In Section~\ref{sec4}\xch{ we}{. we} propose
	functionals on the space of differential one-forms on a spin$_{c}$ manifold
	and prove that these functionals provide densities of the respective dual
	metric and Einstein tensors (\reftext{Theorem \ref{thm4.1}}).
	
	In Section~\ref{sec5} we propose an extension of these functionals to the noncommutative
	realm, focusing in Section~\ref{51} on Laplacian on noncommutative tori,
	where outer derivations are interpreted as vector fields. We prove that
	the spectral Einstein tensor vanishes identically for the conformally rescaled
	geometry of the noncommutative 2-torus (\xch{\reftext{Proposition~\ref{prop5.2}}}{Theorem 5.2}) and provide an explicit
	and compact formula for the spectral Einstein and metric tensors for the
	4-torus. In Section~\ref{52} we propose the generalisation of the spectral
	metric and Einstein functionals on noncommutative differential forms for
	spectral triples on conformally rescaled noncommutative 2 and 4-tori, showing
	that also the contravariant spectral Einstein functional vanishes for the
	conformally rescaled spectral triple on noncommutative 2-torus (\reftext{Proposition~\ref{prop5.10}}) and provide an explicit formula for the spectral metric and Einstein
	functionals for the noncommutative 4-torus. Furthermore, we discuss also
	these functionals for finitely summable and regular spectral triples and
	illustrate how these functionals behave under tensoring a finite summable
	regular spectral triple with the simplest non-trivial finite triple on
	$\mathbb{C}^{2}$.
	
	The Appendix provides a few formulae and computations of pseudodifferential
	symbols, which are used in proofs.
	
	\subsection{Notation}
	\label{13}
	
	Throughout the article we work with a closed, orientable Riemannian manifold
	of dimension $n=2m$, with a fixed metric $g$, while in the case of functionals
	on differential forms we assume the existence of spin$_{c}$ structure and
	fix the spinor bundle. We denote by $\gamma ^{a}$ the matrices that satisfy
	$\gamma ^{a} \gamma ^{b} + \gamma ^{b} \gamma ^{a}=0$ if $a\not =b$, and
	by $(\gamma ^{a})^{2}=1$ for $a,b=1,\ldots ,n$.
	
	We denote the Einstein tensor on $M$ by $G$. The spectral metric and Einstein
	functionals on vector fields we denote by
	${\mathcal{g}}^{P}, {\mathcal{G}}^{P}$, where $P$ is an appropriate Laplace-type
	operator, whereas for the spectral metric and Einstein functionals on differential
	forms we use the notation ${\mathcal{g}}_{D}$ and
	${\mathcal{G}}_{D}$, where $D$ is a Dirac-type operator.
	
	The proofs are all based on the technique of normal coordinates
	${\mathbf{x}}$ around a fixed point on the manifold and the expansion of all
	geometric objects (metric, orthonormal fames, connection, vector fields,
	etc.) up to the relevant order in ${\mathbf{x}}$.
	
	\section{The metric and the Einstein tensor densities}
	\label{sec2}
	
	We consider an even-dimensional compact Riemannian manifold $M$ with components
	of the metric $g$ given in chosen local coordinates by $g_{ab}$. The Laplace
	operator, which is densely defined on $L^{2}(M,vol_{g})$, is expressed
	as
	%
	\begin{equation}
		\Delta = - \frac{1}{\sqrt{\text{det}(g)}} \partial _{a} \bigl( \sqrt{
			\text{det}(g)} g^{ab} \partial _{b} \bigr),
		\label{eq2.1}
	\end{equation}
	where $g^{ab}$ is the inverse of the matrix $g_{ab}$ and we use here and
	in the following the summation convention over repeated indices.
	
	The symbols of the differential operator $\Delta $ are:
	%
	\begin{equation}
		\label{LapSym}
		\begin{aligned}
			{\mathfrak a}_{2} =& g^{ab} \xi _{a} \xi _{b}, \qquad {\mathfrak a}_{1}
			= & \frac{-i}{\sqrt{\text{det}(g)}} \partial _{a} \bigl( \sqrt{\text{det}(g)}
			g^{ab} \bigr) \xi _{b}, \qquad {\mathfrak a}_{0}=0.
		\end{aligned}
		%
	\end{equation}
	As the next step, we will conveniently express the symbols using centered
	normal coordinates ${\mathbf{x}}$ \cite{MSV99} around a fixed point of
	$M$ with ${\mathbf{x}}=0$.
	
	\subsection{Laplace operator and its powers in normal coordinates}
	\label{sec2.1}
	
	Let us recall that in the normal coordinates the metric has a Taylor expansion:
	%
	\begin{equation}
		g_{ab} = \delta _{ab} - \frac{1}{3} R_{acbd} x^{c} x^{d} + o({\mathbf{x^{2}}}),
		\label{gNorm}
	\end{equation}
	and
	%
	\begin{equation}
		\sqrt{\text{det}(g)} = 1 - \frac{1}{6} \mathrm{Ric}_{ab} x^{a} x^{b} + o({
			\mathbf{x^{2}}}),
		\label{volNorm}
	\end{equation}
	where $R_{acbd}$ and $\mathrm{Ric}_{ab}$ are the components of the Riemann
	and Ricci tensor, respectively, at the point with ${\mathbf{x}}=0$ and we use
	the notation $o({\mathbf{x^{k}}})$ to denote that we expand a function up to
	the polynomial of order $k$ in the normal coordinates. The inverse metric
	is
	%
	\begin{equation}
		g^{ab} = \delta _{ab} + \frac{1}{3} R_{acbd} x^{c} x^{d} + o({\mathbf{x^{2}}}),
		\label{gInverseNorm}
	\end{equation}
	where on the right-hand side we still can keep the lower indices
	$a,b$ as at the point with ${\mathbf{x}}=0$ the metric is standard Euclidean
	and so the tensor indices are lowered and raised by the Kronecker symbols
	$\delta _{ab}$ and $\delta ^{ab}$.
	
	Consequently, the symbols of the Laplace operator in normal coordinates
	are
	%
	\begin{equation}
		\label{LapSymNorm}
		\begin{aligned}
			{\mathfrak a}_{2} =& \bigl( \delta _{ab} + \frac{1}{3} R_{acbd} x^{c} x^{d}
			\bigr) \xi _{a} \xi _{b} + o({\mathbf{x^{2}}}),
			\\
			{\mathfrak a}_{1} = & \frac{2i}{3} \mathrm{Ric}_{ab} x^{a} \xi _{b} + o({
				\mathbf{x^{2}}}).
		\end{aligned}
		%
	\end{equation}
	Then, by a straightforward application of \reftext{(\ref{gInverseNorm})}, \reftext{(\ref{gNorm})}, \reftext{(\ref{volNorm})}
	to \reftext{\eqref{SymPinverse}} one has:
	%
	\begin{lem}
		\label{lem2.1}
		In normal coordinates around a fixed point of the manifold $M$ the symbols
		of the inverse of the Laplace operator read
		%
		\begin{equation}
			\begin{aligned}
				&\mathfrak b_{2} = ||\xi ||^{-4} \bigl( \delta _{ab} - \frac{1}{3} R_{acbd}
				x^{c} x^{d} \bigr) \xi _{a} \xi _{b} + o(\mathbf{x^{2})},
				\\
				& \mathfrak b_{3} = -\frac{2i}{3} \mathrm{Ric}_{ab} x^{a} \xi _{b} ||
				\xi ||^{-4} + o(\mathbf{x)},
				\\
				&\mathfrak b_{4} = \frac{2}{3} \mathrm{Ric}_{ab} \xi _{a} \xi _{b} ||
				\xi ||^{-6} + o(\mathbf{1).
			}\end{aligned}
			\label{normalny_laplasjan}
		\end{equation}
	\end{lem}
	Next, we apply (\reftext{Lemma~\ref{LA1}}) to compute the three leading symbols of
	the powers of the pseudodifferential operator $\Delta ^{-1}$:
	%
	\begin{prop}
		\label{prop2.2}
		The first three leading symbols of the operator $\Delta ^{-k}$,
		$k >0$
		\begin{equation*}
			\sigma (\Delta ^{-k})=\mathfrak c_{2k}+\mathfrak c_{2k+1}+\mathfrak c_{2k+2}+
			\ldots ,
		\end{equation*}
		are given up to order respectively $\mathbf{x^{2}},\mathbf{ x},\mathbf{ 1}$ in normal coordinates
		around a fixed point by,
		%
		\begin{equation}
			\begin{aligned}
				&\mathfrak c_{2k} = ||\xi ||^{-2k-2} \left ( \delta _{ab} -
				\frac{k}{3} R_{ac bd} x^{c} x^{d} \right ) \xi _{a} \xi _{b} + o(\mathbf{x^{2})},
				\\
				&\mathfrak c_{2k+1}=\frac{-2ki}{3||\xi ||^{2k+2}}\mathrm{Ric}_{ab }x^{b}
				\xi _{a} + o(\mathbf{x)},
				\\
				&\mathfrak c_{2k+2}=\frac{k(k+1)}{3||\xi ||^{2k+4}}\mathrm{Ric}_{ab}
				\xi _{a}\xi _{b} + o(\mathbf{1).
			}\end{aligned}
			\label{laplace_normal_powers}
		\end{equation}
	\end{prop}
	\begin{proof}
		We use \reftext{\eqref{normalny_laplasjan}} and \reftext{\eqref{commutative_operator_powers}} keeping only terms with the right order
		in $\mathbf{x}$, which for $\mathfrak c_{2k+2}$ yields only 4 terms instead
		of 10 in \reftext{\eqref{commutative_operator_powers}}, i.e.
		\begin{equation*}
			\mathfrak c_{2k+2}=k (\mathfrak b_{2})^{n-1}\mathfrak b_{4} -i
			\frac{k(k-1)}{2}(\mathfrak b_{2})^{k-2}\partial ^{\xi}_{a}(
			\mathfrak b_{2})\partial _{a}^{x}(\mathfrak b_{3})-
		\end{equation*}
		\begin{equation*}
			-\frac{k(k-1)}{4}(\mathfrak b_{2})^{k-2}\partial ^{\xi}_{a}\partial ^{
				\xi}_{b}(\mathfrak b_{2})\partial ^{x}_{a}\partial ^{x}_{b}(
			\mathfrak b_{2})-
		\end{equation*}
		\begin{equation*}
			-\frac{k(k-1)(k-2)}{6}(\mathfrak b_{2})^{k-3}\partial ^{\xi}_{a}(
			\mathfrak b_{2})\partial ^{\xi}_{b}(\mathfrak b_{2})\partial _{a}^{x}
			\partial _{b}^{x}(\mathfrak b_{2}) + o(\mathbf{1).
		}\end{equation*}
		We simplify it further by taking into account the properties of the Riemann
		tensor,
		\begin{equation*}
			R_{abcd}\xi _{a}\xi _{b}=0=R_{abcd}\xi _{c}\xi _{d},
		\end{equation*}
		leading to the above result.
	\end{proof}
	%
	
	\subsection{Spectral functionals of vector fields}
	\label{sec2.2}
	
	Let $V,W$ be a pair of vector fields on a compact Riemannian manifold
	$M$, of dimension $n=2m$. Using the Laplace operator, we define two functionals
	$ \mathcal{g}^{\Delta}(V,W)$ and $ \mathcal{G}^{\Delta}(V,W)$.
	%
	\begin{thm}
		\label{metricthm}
		The functional:
		\begin{equation*}
			\mathcal{g}^{\Delta}(V,W) := \mathcal{W}\left ( VW \Delta ^{-m-1}
			\right ),
		\end{equation*}
		is a bilinear, symmetric map, whose density is proportional to the metric
		evaluated on the vector fields:
		\begin{equation*}
			{ \mathcal{g} }^{\Delta}(V,W) = -\frac{v_{n-1}}{n} \int _{M} g(V,W)\, vol_{g}.
		\end{equation*}
	\end{thm}
	\begin{proof}
		The product of two vector fields is a differential operator with a symbol
		%
		\begin{equation}
			\sigma (VW)=\mathfrak v_{2}+\mathfrak v_{1}=-V^{a}W^{b}\xi _{a}\xi _{b}+iV^{a}
			\delta _{a}(W^{b})\xi _{b}.
			\label{eq2.9}
		\end{equation}
		Then,
		\begin{equation*}
			\mathcal{W}\left ( VW \Delta ^{-m-1} \right ) = \int _{M} \int _{||
				\xi ||=1}\sigma _{-2m}(VW \Delta ^{-m-1})\, vol_{g},
		\end{equation*}
		where by \reftext{\eqref{laplace_normal_powers}}
		\begin{equation*}
			\sigma _{-2m}(VW \Delta ^{-m-1})=-V^{a}W^{b}\xi _{a}\xi _{b}||\xi ||^{-2m-2}.
		\end{equation*}
		Using integration over $S^{2m-1}$ we get
		\begin{equation*}
			\int _{||\xi ||=1}\sigma _{-2m}(VW \Delta ^{-m-1})=-\frac{v_{n-1}}{n}V^{a}W^{a},
		\end{equation*}
		which ends the proof.
	\end{proof}
	%
	\begin{thm}%
		\label{einsteinthm}
		The functional:
		%
		\begin{equation}
			\mathcal{G}^{\Delta}(V,W) := \mathcal{W}\left ( VW \Delta ^{-m}
			\right ),
			\label{eq2.10}
		\end{equation}
		is a bilinear, symmetric map, whose density is proportional to the Einstein
		tensor $G$ evaluated on the two vector fields:
		\begin{equation*}
			\mathcal{G}^{\Delta}(V,W) = {\frac{v_{n-1}}{6}} \int _{M} G(V,W)\, vol_{g}.
		\end{equation*}
	\end{thm}
	\begin{proof}
		Similarly to the previous theorem, we need to compute the symbol of order
		$-2m$ of the pseudodifferential operator $VW\Delta ^{-m}$. Using \reftext{\eqref{composition}} we have
		%
		\begin{equation}
			\begin{aligned}
				\sigma _{-2m}(VW\Delta ^{-m}) &= \mathfrak v_{2} \mathfrak c_{2m+2}+
				\mathfrak v_{1}\mathfrak c_{2m+1} -i\partial _{a}\mathfrak v_{2}
				\delta _{a}\mathfrak c_{2m+1}
				\\
				& \quad -i \partial _{a}\mathfrak v_{1}\delta _{a}\mathfrak c_{2m} -
				\frac{1}{2}\partial _{a}\partial _{b}\mathfrak v_{2}\delta _{a}
				\delta _{b}\mathfrak c_{2m}.
			\end{aligned}
			\label{eq2.11}
		\end{equation}

		As the normal coordinates $\mathbf{x}$ are $0$ at the (arbitrary) fixed point of the manifold, we are interested only in terms that
		do not vanish at ${\mathbf{x}}=0$. Since both $\mathfrak c_{2m+1}$ and
		$\delta _{a} \mathfrak c_{2m}$ vanish at ${\mathbf{x}}=0$ we are left only with terms that depend on $\mathfrak v_{2}$. We explicitly compute
		%
		\begin{equation}
			\begin{aligned}
				\mathfrak v_{2}\mathfrak c_{2m+2} &=-\frac{m(m+1)}{3||\xi ||^{2m+4}} V^{a}W^{b}
				\mathrm{Ric}_{cd}\xi _{a}\xi _{b}\xi _{c}\xi _{d} + o(\mathbf{1)},
				\\
				-i\partial _{a}\mathfrak v_{2}\delta _{a}\mathfrak c_{2m+1}& =-i(-V^{a}W^{b}
				\xi _{b}-V^{b}W^{a}\xi _{b})\frac{-2mi}{3||\xi ||^{2m+2}}
				\mathrm{Ric}_{ca}\xi _{c} + o(\mathbf{1)
				}\\
				& = \frac{2m}{3||\xi ||^{2m+2}}V^{a}W^{b}\xi _{c}(\xi _{a}
				\mathrm{Ric}_{bc}+\xi _{b}\mathrm{Ric}_{ac})+o(\mathbf{1)},\\
				-\frac{1}{2}\partial _{a}\partial _{b}\mathfrak v_{2}\delta _{a}
				\delta _{b}\mathfrak c_{2m}& = \frac{1}{2}(V^{a}W^{b}+V^{b}W^{a})
				\frac{-m}{3||\xi ||^{2m+2}}(R_{cadb} +R_{cbda}) \xi _{c}\xi _{d} + o(
				\mathbf{1})
				\\
				& =-\frac{2m}{3||\xi ||^{2m+2}}V^{a}W^{b}\xi _{c}R_{cadb}\xi _{d}+o(
				\mathbf{1).
			}\end{aligned}
			\label{eq2.12}
		\end{equation}
		As a result,
		%
		\begin{equation}
			\begin{aligned}
				\sigma _{-2m}(VW\Delta ^{-m})&=-\frac{m(m+1)}{3||\xi ||^{2m+4}} V^{a}W^{b}
				\mathrm{Ric}_{cd}\xi _{a}\xi _{b}\xi _{c}\xi _{d}
				\\
				&\quad +\frac{2m}{3||\xi ||^{2m+2}}V^{a}W^{b}\xi _{c}(\xi _{a}
				\mathrm{Ric}_{bc}+\xi _{b}\mathrm{Ric}_{ac}-R_{cadb}\xi _{d})+o(\mathbf{1}).
			\end{aligned}
			\label{VWDelta_symbol}
		\end{equation}
		Integrating \reftext{\eqref{VWDelta_symbol}} over $S^{2m-1}$ and substituting ${\mathbf{x}}=0$ we have, for the first term,
		\begin{equation*}
			\begin{aligned}
				\int _{||\xi ||=1}-\frac{m(m+1)}{3||\xi ||^{2m+4}} &V^{a}W^{b}
				\mathrm{Ric}_{cd}\xi _{a}\xi _{b}\xi _{c}\xi _{d} =
				\\
				& =\frac{v_{n-1}}{2m(2m+2)}(\delta _{ab}\delta _{cd}+\delta _{ac}
				\delta _{bd} +\delta _{ad}\delta _{bc})\frac{-m(m+1)}{3} V^{a} W^{b}
				\mathrm{Ric}_{cd}
				\\
				& =-\frac{v_{n-1}}{12}V^{a}W^{a}R-\frac{v_{n-1}}{6}V^{a}W^{b}
				\mathrm{Ric}_{ab}\xch{,}{.}
			\end{aligned}
		\end{equation*}
		and for the second term,
		\begin{equation*}
			\begin{aligned}
				\int _{||\xi ||=1}\frac{2m}{3||\xi ||^{2m+2}} & V^{a}W^{b}\xi _{c}(
				\xi _{a}\mathrm{Ric}_{bc}+\xi _{b}\mathrm{Ric}_{ac}-R_{cadb}\xi _{d}) =
				\\
				& =\frac{v_{n-1}}{2m}\cdot \frac{2m}{3}V^{a}W^{b}(\delta _{ac}
				\mathrm{Ric}_{bc}+\delta _{bc}\mathrm{Ric}_{ac}-\delta _{cd}R_{cadb})
				\\
				& =\frac{v_{n-1}}{3}V^{a}W^{b}\mathrm{Ric}_{ab}.
			\end{aligned}
		\end{equation*}
		Combining them together and integrating over the manifold we get the result, which is obviously symmetric and bilinear.
	\end{proof}
	%
	\begin{rem}%
		\label{localgG}
		A localized version of the functionals $\mathcal{g}^{\Delta}$ and
		${\mathcal G}^{\Delta}$ is automatic since they satisfy,
		%
		\begin{equation}
			\mathcal{W}( f VW \Delta ^{-n-1}) = \mathcal{g}^{\Delta}(fV,W)=
			\mathcal{g}^{\Delta}(V,fW),
			\label{eq2.14}
		\end{equation}
		%
		\begin{equation}
			\mathcal{W}( f VW \Delta ^{-n}) = \mathcal{G}^{\Delta}(fV,W)=
			\mathcal{G}^{\Delta}(V,fW)
			\label{eq2.15}
		\end{equation}
		for all $f\in C^{\infty}(M)$.
	\end{rem}
	We will call $\mathcal{g}^{\Delta}$ and $\mathcal{G}^{\Delta}$ respectively
	\textit{metric} and \textit{Einstein} (spectral) functionals.
	
	\section{Spectral functionals for the operators of Laplace type}
	\label{sec3}
	
	In this section, we will demonstrate how the results vary if one passes
	to the Laplace-type operators, acting on sections of a vector bundle
	$V$ of rank $\text{rk}(V)$. These operators generalize the scalar Laplacian
	in the sense that they have the same principal symbol (times the matrix
	unit), yet they may contain both some nontrivial connections and torsion.
	
	For this purpose we assume that there is a connection $\nabla $ on the
	vector bundle $V$, i.e. for any vector field $X$ on $M$, we have a covariant
	derivative $\nabla _{X}$ on the module of smooth sections of $V$. Using
	the notation $\nabla _{a}:=\nabla _{\partial _{a}}$ in local normal coordinates
	around a fixed point on the manifold we have:
	\begin{equation*}
		\nabla _{a}= \partial _{a} - {\mathbf{T}}_{a},
	\end{equation*}
	where each ${\mathbf{T}}_{a}$ is a $C^{\infty}(M)$ endomorphism of the sections
	of $V$. Using normal coordinates and expanding ${\mathbf{T}}_{a}$ around
	${\mathbf{x}}=0$ we have,
	\begin{equation*}
		{\mathbf{T}}_{a}(x) = T_{a} + T_{ab} x^{b} + o({\mathbf{x}}).
	\end{equation*}
	Thus, we can write the symbol of the generalized Laplace operator
	\begin{equation*}
		\Delta _{T} = - g^{ab} ( \nabla _{a} \nabla _{b} - \Gamma ^{c}_{ab}
		\nabla _{c}),
	\end{equation*}
	in the normal coordinates as the sum of
	%
	\begin{equation}
		\label{LapTF0}
		\begin{aligned}
			{\mathfrak a}_{2} =& \bigl( \delta _{ab} + \frac{1}{3} R_{acbd} x^{c} x^{d}
			\bigr) \xi _{a} \xi _{b} +o({\mathbf{x^{2}}}),
			\\
			{\mathfrak a}_{1} = & \frac{2i}{3} \mathrm{Ric}_{ab} x^{a} \xi _{b} + (2
			i T_{a} \xi _{a} + 2i T_{ab} x^{b} \xi _{a}) +o({\mathbf{x}}),
			\\
			{\mathfrak a}_{0} = & T_{aa} - T_{a} T_{a} + o({\mathbf{1}}).
		\end{aligned}
		%
	\end{equation}
	Then the parts of order $-2, -3, -4$ of the inverse of $\Delta _{T}$ expanded in ${\bf x}$ up to the order, respectively $2,1,0$, are
	%
	\begin{equation}
		\begin{aligned}
			&\mathfrak b_{2} = ||\xi ||^{-4} \bigl( \delta _{ab} - \frac{1}{3} R_{acbd}
			x^{c} x^{d} \bigr) \xi _{a} \xi _{b} + o(\mathbf{x^{2})},
			\\
			& \mathfrak b_{3} = -\frac{2i}{3} \mathrm{Ric}_{ab} x^{a} \xi _{d} ||
			\xi ||^{-4} -\bigl( 2 i T_{a} \xi _{a} + 2i T_{ab} x^{b} \xi _{a}
			\bigr) ||\xi ||^{-4} + o({\mathbf{x}}),
			\\
			&\mathfrak b_{4} = \frac{2}{3} \mathrm{Ric}_{ab} \xi _{a} \xi _{b} ||
			\xi ||^{-6} - 4 T_{a} T_{b} \xi _{a} \xi _{b}||\xi ||^{-6}
			\\
			& \qquad - (T_{aa} -T_{a} T_{a} ) ||\xi ||^{-4} + 4 T_{ab} \xi _{a}
			\xi _{b} ||\xi ||^{-6} + o(\mathbf{1) .
		}\end{aligned}
		\label{LapTF1}
	\end{equation}
	%
	\begin{thm}
		\label{thm3.1}
		The functional
		\begin{equation*}
			\mathcal{g}^{\Delta _{T}}(V,W):= \mathcal{W}( \nabla _{V}\nabla _{W}
			\Delta _{T}^{-n-1})
		\end{equation*}
		does not depend on the connection ${\mathbf{T}}$, and
		\begin{equation*}
			\mathcal{g}^{\Delta _{T}}(V,W)= \text{rk}(V)\, \mathcal{g}^{\Delta}(V,W).
		\end{equation*}
	\end{thm}
	\begin{proof}
		This is evident, since the principal symbol does not depend on
		${\mathbf{T}}$.
	\end{proof}
	%
	\begin{thm}%
		\label{lifteinstein}
		The functional
		\begin{equation*}
			\mathcal{G}^{\Delta _{T}}(V,W):= \mathcal{W}( \nabla _{V}\nabla _{W}
			\Delta _{T}^{-n})
		\end{equation*}
		is equal to
		%
		\begin{equation}
			\mathcal{G}^{\Delta _{T}}(V,W)= {\frac{v_{n-1}}{6}} \text{rk}(V)\,
			\int _{M} G(V,W)\, vol_{g} + {\frac{v_{n-1}}{2}} \int _{M} F(V,W)\, vol_{g},
			\label{eq3.3}
		\end{equation}
		where
		\begin{equation*}
			F(V,W) = \text{Tr\ } V^{a} W^{b} F_{ab},
		\end{equation*}
		and $F_{ab}$ is the curvature tensor of the connection ${\mathbf{T}}$. For
		${\mathbf{T}}=0$ it is equal to $\text{rk}(V)\mathcal{G}^{\Delta}(V,W)$.
	\end{thm}
	\begin{proof}
		First we compute the leading symbols of $(\Delta _{T})^{-m}$ up to the
		appropriate order in \textbf{x},
		%
		\begin{equation}
			\begin{aligned}
				&\mathfrak c_{2m} = ||\xi ||^{-2m-2} \left ( \delta _{ab} -
				\frac{m}{3} R_{ajbk} x^{j} x^{k} \right ) \xi _{a} \xi _{b} + o(\mathbf{x^{2})},
				\\
				&\mathfrak c_{2m+1}=\frac{-2 mi}{3} ||\xi ||^{-2m-2} \mathrm{Ric}_{ak}
				x^{k} \xi _{a} - 2 m i ||\xi ||^{-2m-2} \bigl( T_{a} \xi _{a} + T_{ab}
				x^{b} \xi _{a} \bigr) + o(\mathbf{x)
				}\\
				&\mathfrak c_{2m+2}=\frac{m(m+1)}{3} ||\xi ||^{-2m-4} \mathrm{Ric}_{ab}
				\xi _{a}\xi _{b}
				\\
				& \qquad \qquad - 2 m(m+1) ||\xi ||^{-2m-4} \, T_{a} T_{b}\, \xi _{a}
				\xi _{b} + m (T_{a} T_{a} - T_{aa}) ||\xi ||^{-2m-2}
				\\
				& \qquad \qquad +2m (m+1) ||\xi ||^{-2m-4} \, T_{ab} \, \xi _{a} \xi _{b}
				+ o(\mathbf{1).
			}\end{aligned}
			\label{LapTF3}
		\end{equation}
		Expanding in a similar way
		\begin{equation*}
			\nabla _{V} = V^{a}(\partial _{a} - T_{a} - T_{ab} x^{b}) + o({\mathbf{x}}),
			\qquad \nabla _{W} = W^{a}(\partial _{a} - T_{a} - T_{ab} x^{b}) + o({
				\mathbf{x}}),
		\end{equation*}
		for vector fields $V$ and $W$, we compute only the terms of the symbol
		of $\nabla _{V}\nabla _{W} \Delta _{T}^{-m}$ of order $-2m$, which depend
		on $T_{a}$ and $T_{ab}$ at $x=0$, as the remaining terms would be identical
		to the case considered earlier,
		%
		\begin{equation}
			\begin{aligned}
				\sigma _{-2m}( \nabla _{V}\nabla _{W} \Delta _{T}^{-m}) &= \bigl( 2m(m+1)
				T_{a} T_{b}\, \xi _{a} \xi _{b} + m T_{aa} ||\xi ||^{2} - m T_{a} T_{a}
				||\xi ||^{2}
				\\
				& \quad -2m(m+1) T_{ab}\, \xi _{a} \xi _{b} \bigr) V^{j} W^{k} \xi _{j}
				\xi _{k} || \xi ||^{-2m-4}
				\\
				& \quad +2 m V^{a} (\partial _{a} W^{b} ) T_{c} \xi _{b} \xi _{c} ||
				\xi ||^{-2m-2}
				\\
				& \quad - 2 m (V^{a} W^{b} + V^{b} W^{a}) T_{a} T_{c} \xi _{b} \xi _{c}
				||\xi ||^{-2m-2}
				\\
				& \quad + V^{a} W^{b} T_{a} T_{b} ||\xi ||^{-2m} - V^{a} W^{b} T_{ba} ||
				\xi ||^{-2m}
				\\
				& \quad + 2 m T_{ab} (V^{c} W^{b} + W^{c} V^{b}) \xi _{a} \xi _{c} ||
				\xi ||^{-2m-2}
				\\
				& \quad - V^{a} (\partial _{a} W^{b} ) T_{b} ||\xi ||^{-2m}.
			\end{aligned}
			\label{VWLapTF1}
		\end{equation}
		Integrating it over $S^{2m-1}$ and setting ${\mathbf{x}}=0$ we have the following
		result for the density of the Wodzicki residue (before taking the trace
		over endomorphisms)
		%
		\begin{equation}
			\begin{aligned}
				\mathcal{w}(\nabla _{V}\nabla _{W} \Delta _{T}^{-m}) &= v_{n-1}
				\biggl( \frac{1}{2} V^{a} W^{b} \bigl( T^{2} \delta _{ab} + T_{a} T_{b}
				+T_{b} T_{a} + T_{cc} \delta _{ab}
				\\
				&\quad - T^{2} \delta _{ab} - T_{cc} \delta _{ab} -T_{ab} - T_{ba}
				\bigr) + V^{a} (\partial _{a} W^{b}) T_{b}
				\\
				&\quad - V^{a} W^{b} (T_{a} T_{b} +T_{b} T_{a}) + V^{a} W^{b} T_{a} T_{b}
				- V^{a} W^{b} T_{ba}
				\\
				& \quad + V^{a} W^{b} (T_{ab} + T_{ba}) - V^{a} (\partial _{a} W^{b}) T_{b}
				\biggr)
				\\
				& = \frac{1}{2} v(S^{n-1}) V^{a} W^{b} \bigl( T_{ab} - T_{ba} + T_{a} T_{b}
				- T_{b} T_{a} \bigr)
				\\
				& = \frac{1}{2} v(S^{n-1}) V^{a} W^{b} F_{ab} ,
			\end{aligned}
			\label{eq3.6}
		\end{equation}
		where $F_{ab}$ is the curvature tensor of the connection $ {\mathbf{T}}$, as
		indeed in the normal coordinates:
		\begin{equation*}
			F_{ab} = - ( \partial _{a} {\mathbf{T}}_{b} - \partial _{b} {\mathbf{T}}_{a}) + [{
				\mathbf{T}}_{a}, {\mathbf{T}}_{b}] = T_{ab} - T_{ba} + [T_{a}, T_{b}] + o(\mathbf{1)},
		\end{equation*}
		which finishes the proof.
	\end{proof}
	Note that the functionals $\mathcal{g}^{\Delta _{T}}$ and
	$\mathcal{G}^{\Delta _{T}}$ are automatically `localised' as in \reftext{Remark~\ref{localgG}}.
	
	We finish the section by demonstrating that any 0-order perturbation of
	the Laplace (type) operator $\Delta _{T}$ does not modify the metric functional,
	whereas the Einstein functional is modified by a term whose density involves
	the metric functional multiplied by the trace of the 0-order term.
	%
	\begin{lem}%
		\label{laplscal}
		For the Laplace type operator $\Delta _{T,E}:= \Delta _{T}+ E$, where
		$E$ is an endomorphism of the vector bundle, the metric functional does
		not depend on $E$,
		%
		\begin{equation}
			\mathcal{g}^{\Delta _{T,E}}(V,W)= \mathcal{g}^{\Delta _{T}}(V,W)\xch{,}{.}
			\label{eq3.7}
		\end{equation}
		whereas the Einstein functional
		\begin{equation*}
			\mathcal{G}^{\Delta _{T,E}}(V,W):= \mathcal{W}( \nabla _{V}\nabla _{W}
			\Delta _{T,E}^{-m})
		\end{equation*}
		reads,
		%
		\begin{equation}
			\mathcal{G}^{\Delta _{T,E}}(V,W)= \mathcal{G}^{\Delta _{T}}(V,W) +
			\frac{1}{2} \int _{M} (\text{Tr\ } E) \, g(V,W)~vol_{g}.
			\label{eq3.8}
		\end{equation}
	\end{lem}
	\begin{proof}
		The first statement is obvious. For the second statement by computing the
		symbols of $\Delta ^{-m}$ we first see from the formula \reftext{\eqref{SymPinverse}} that $E$ enters the symbol of order $-4$ linearly and,
		consequently, using \reftext{\eqref{commutative_operator_powers}} the only additional
		term that appears in the symbol of $\Delta ^{-m}$ would appear in order
		$-2m-2$ as
		\begin{equation*}
			- m E ||\xi ||^{-2m-2}.
		\end{equation*}
		The only term depending on $E$ that arises in the contribution to the respective
		symbol of order $-2m$ of the product
		$\nabla _{V}\nabla _{W} \Delta ^{-m}$ would then be
		\begin{equation*}
			m E V^{a} W^{b} \xi _{a} \xi _{b} ||\xi ||^{-2m-2},
		\end{equation*}
		which after integrating over the sphere, taking trace and using \reftext{(\ref{gNorm})}
		gives
		\begin{equation*}
			\frac{1}{2}\, (\text{Tr\ } E) g(V,W).
			\qedhere \end{equation*}
	\end{proof}
	\noindent
	As follows from \cite{Gi04,Gi84} the operators $\Delta _{T,E}$ are the
	most general Laplace-type operators on $M$.
	
	\subsection{Spectral functionals for the spin Laplacian}
	\label{sec3.1}
	
	A particularly interesting example is the application of the above result
	to the case of the spinor bundle of rank $2^{m}$ (assuming that the manifold
	$M$ has a spin structure, which we fix).
	
	Recall that in terms of a (local) basis $e_{i}$ of orthonormal vector fields
	on $M$ the Levi-Civita covariant derivative reads
	%
	\begin{equation}
		\label{connform1}
		\nabla _{e_{i}}e_{j} = \alpha _{ijk} e_{k},
	\end{equation}
	where $\alpha _{ijk}$ is the Levi-Civita connection, expressed explicitly
	in the basis $e_{i}$ through the structure constants $c_{ijk}$ of the commutators
	of orthonormal vector fields,
	%
	\begin{equation}
		\label{connform}
		\alpha _{ijk}:= \frac{1}{2} (c_{ijk}+ c_{kij} + c_{kji}), \qquad [ e_{i},
		e_{j} ] = c_{ijk} e_{k} .
	\end{equation}
	If $M$ is a spin manifold, then the lift of the Levi-Civita covariant derivative
	to Dirac spinor fields reads
	%
	\begin{equation}
		\nabla ^{(s)}_{e_{i}} = e_{i} - \frac{1}{4}\alpha _{ijk}\gamma ^{j}
		\gamma ^{k},
		\label{eq3.11}
	\end{equation}
	with $\gamma ^{j}$ as in Section~\ref{13}. The spinorial Laplace operator
	is
	%
	\begin{equation}
		\Delta ^{(s)}:= \nabla ^{(s)*} \nabla ^{(s)} = -\nabla ^{(s)}_{e_{i}}
		\nabla ^{(s)}_{e_{i}}+\nabla ^{(s)}_{\nabla _{e_{i}}e_{i}} =-\nabla ^{(s)}_{e_{i}}
		\nabla ^{(s)}_{e_{i}}+\alpha _{iij}\nabla ^{(s)}_{e_{j}}.
		\label{eq3.12}
	\end{equation}
	In the normal coordinates around a fixed point of the manifold one has
	then
	%
	\begin{equation}
		e_{i}= \frac{\partial}{\partial x^{i}} - \frac{1}{6} R_{ijk\ell}\, x^{j}
		x^{k} \frac{\partial}{\partial x^{\ell}} + o(\mathbf{x^{2})
			\label{eq3.13}
	}\end{equation}
	and
	%
	\begin{equation}
		\alpha _{ijk} = - \frac{1}{2} R_{\ell ijk} x^{\ell }+ \mathbf{o(}\mathbf{x)},\mathbf{
			\label{eq3.14}
	}\end{equation}
	and therefore,
	%
	\begin{equation}
		\begin{aligned}
			\Delta ^{(s)} = -\partial _{i}\partial _{i} + \frac{1}{3} R_{ijk\ell}
			\, x^{j}x^{k} \partial _{i} \partial _{\ell }+ o(\mathbf{x^{2})
			}\\
			+ \frac{2}{3}R_{ij}\, x^{i}\partial _{j} + \frac{1}{4}R_{i\ell jk}\, x^{
				\ell }\gamma ^{j} \gamma ^{k}\partial _{i} + o(\mathbf{x)
			}\\
			+ o(\mathbf{1)},
		\end{aligned}
		\label{eq3.15}
	\end{equation}
	where in the three lines we collected terms of different order, expanding
	them up to the same order in normal coordinates $\mathbf{x}$ as their order.
	
	By comparing the symbol with \reftext{\eqref{LapTF0}} it is easy to identify then
	$\Delta ^{(s)}$ as a Laplace-type operator $\Delta _{T}$, with the respective
	expansion of the spin connection,
	\begin{equation*}
		T_{a} = 0, \qquad T_{ab} = \frac{1}{8} R_{abjk} \gamma ^{j} \gamma ^{k}.
	\end{equation*}
	We have then
	%
	\begin{prop}
		\label{prop3.4}
		The metric and Einstein functional associated to the spin Laplacian are
		proportional to the functionals of the scalar Laplacian,
		%
		\begin{equation}
			\begin{aligned}
				\mathcal{g}^{\Delta ^{(s)}}(V,W)&:= \mathcal{W}( \nabla ^{(s)}_{V}
				\nabla ^{(s)}_{W} (\Delta ^{(s)})^{-n-2}) = 2^{m} \mathcal{g}^{\Delta}(V,W),
				\\
				\mathcal{G}^{\Delta ^{(s)}}(V,W)&:= \mathcal{W}( \nabla ^{(s)}_{V}
				\nabla ^{(s)}_{W} (\Delta ^{(s)})^{-n}) = 2^{m} \mathcal{G}^{\Delta}(V,W).
			\end{aligned}
			\label{eq3.16}
		\end{equation}
	\end{prop}

	\begin{proof}
		This is a consequence of \reftext{Theorem~\ref{lifteinstein}}. While the result for
		the metric functional is obvious for the Einstein functional we see the
		trace of the additional term, arising from the curvature of the connection
		vanishes,
		\begin{equation*}
			\frac{1}{4} \text{Tr\ } V^{a} W^{b} \bigl( R_{abjk} \gamma ^{j}
			\gamma ^{k} - R_{bajk} \gamma ^{j} \gamma ^{k} \bigr) = 0,
		\end{equation*}
		due to the skew \xch{symmetry}{symetry} of the Riemann tensor in the two last indices.
	\end{proof}
	
	\subsection{Spectral functionals for the Dirac operator}
	\label{sec3.2}
	
	An immediate application of \reftext{Lemma~\ref{laplscal}} is the computation of
	the Einstein functional using the Dirac operator instead of the Laplacian.
	The Dirac operator is, in a local basis of orthonormal frames, a first-order
	differential operator
	%
	\begin{equation}
		D = i \gamma ^{j} \nabla ^{(s)}_{e_{j}},
		\label{Dirac}
	\end{equation}
	and its square $D^{2}$ differs from the spinorial Laplacian only by a quarter
	of the scalar of curvature
	\begin{equation*}
		D^{2} = \Delta ^{(s)} + \frac{1}{4} R.
	\end{equation*}
	More generally, if we consider a spin$_{c}$ structure and the Dirac operator
	twisted by a $U(1)$-connection,
	\begin{equation*}
		D_{A} = D + A,
	\end{equation*}
	the respective formula reads,
	\begin{equation*}
		D_{A}^{2} = \Delta ^{(s)} + \frac{1}{4} R + F,
	\end{equation*}
	where
	\begin{equation*}
		F = \gamma ^{j} \gamma ^{k} F_{jk},
	\end{equation*}
	and $F_{jk}$ is the curvature of $A$. Then as the consequence of \reftext{Lemma~\ref{laplscal}}
	we have,
	%
	\begin{prop}
		\label{prop3.5}
		The spectral metric and Einstein functionals associated with the Dirac
		operator $D_{A}$ do not depend on the connection $A$ and read,
		%
		\begin{equation}
			\begin{aligned}
				\mathcal{g}^{D_{A}^{2}}(V,W) &:= \mathcal{W}(\nabla ^{(s)}_{V}
				\nabla ^{(s)}_{W} |D_{A}|^{-n-2}) = 2^{m} \mathcal{g}^{\Delta}(V,W),
				\\
				\mathcal{G}^{D_{A}^{2}}(V,W) &:= \mathcal{W}(\nabla ^{(s)}_{V}
				\nabla ^{(s)}_{W} |D_{A}|^{-n}) = 2^{m} \left ( \mathcal{G}^{\Delta}(V,W)
				+ \frac{1}{8} {\mathcal R}(V,W) \right ),
			\end{aligned}
			\label{eq3.18}
		\end{equation}
		where
		\begin{equation*}
			{\mathcal R}(V,W) = v_{n-1} \int _{M} R(g) \, g(V,W)\, vol_{g}.
		\end{equation*}
	\end{prop}
	%
	
	\section{Metric and Einstein functionals of differential forms}
	\label{sec4}
	
	In differential geometry besides functionals on vector fields, one can
	alternatively consider functionals on the dual bimodule of one-forms. In
	this section, we will investigate whether the Einstein tensor (or, more
	precisely, its contravariant version) can be obtained from such functional
	using the spectral methods. For this purpose, we need to represent differential
	forms as differential operators, and a suitable way is to employ Clifford
	modules. We assume thus that $M$ is a $n=2m$ dimensional spin$_{c}$ manifold
	and use the Clifford representation of one-forms as $0$-order differential
	operators, that is, endomorphisms of a rank $2^{m}$ spinor bundle.
	
	As generating endomorphisms, we can work with Clifford multiplication by
	the local coframe basis $e^{j}$ dual to the orthonormal oriented frame
	basis $e_{j}$, which simply amounts to multiplication by (constant) gamma
	matrices $\gamma ^{j}$ as in Section~\ref{13}. Thus, let $v,w$, with the
	components with respect to local coordinates $v_{a}$ and $w_{a}$, respectively,
	be two differential forms represented in such a way as endomorphisms (matrices)
	$\hat v$ and $\hat w$ on the spinor bundle.
	%
	\begin{thm}
		\label{thm4.1}
		The following spectral functionals of one-forms on a spin-c manifold
		$M$ of dimension $n$
		%
		\begin{equation}
			\begin{aligned}
				\mathcal{g}_{D}(v,w) &:= \mathcal{W}\bigl( \hat v \hat w D^{-n}
				\bigr),
				\\
				\mathcal{G}_{D}(v,w) &:= \mathcal{W}\bigl( \hat v (D\hat w+\hat wD )
				D^{-n+1} \bigr)
				\\
				&\,\,=\mathcal{W}\bigl( (D\hat v+\hat vD ) \hat wD^{-n+1} \bigr),
			\end{aligned}
			\label{eq4.1}
		\end{equation}
		read
		%
		\begin{equation}
			\label{H13}
			\begin{aligned}
				& \mathcal{g}_{D}(v,w) = 2^{m} v_{n-1} \int _{M} g(v,w)~vol_{g},
				\\
				& \mathcal{G}_{D}(v,w) = 2^{m} \frac{v_{n-1}}{6} \int _{M} G(v,w)~vol_{g}
				,
			\end{aligned}
			%
		\end{equation}
		where $g(v,w) = g^{ab} v_{a} w_{b}$ and
		$G(v,w) = \bigl( \text{Ric}^{ab} - \frac{1}{2} R g^{ab} \bigr) v_{a} w_{b}$,
		using the expressions $v=v_a dx^a$, $w=w_b dx^b$ in any local coordinates.
	\end{thm}
	\begin{proof}
		The proof of the formula for the metric functional $\mathcal{g}_{D}$ is
		easy and we skip it, concentrating on the Einstein functional, splitting
		it into two parts as follows:
		\begin{equation*}
			\mathcal{G}_{1}(v,w) = \mathcal{W}\bigl( \hat v D\hat w DD^{-n}
			\bigr),
		\end{equation*}
		and
		\begin{equation*}
			\mathcal{G}_{2}(v,w) = \mathcal{W}\bigl( \hat v \hat w D^{-n+2}
			\bigr).
		\end{equation*}
		We start with the first. Let us again work with normal coordinates
		$x$ around a fixed point on the manifold $M$ with ${\mathbf{x}} =0$. We can
		rewrite the first three leading symbols of $ D^{-n} =D^{-2m}$ (using \reftext{\eqref{LapTF3}} and the Lichnerowicz formula),
		%
		\begin{equation}
			\begin{aligned}
				&\mathfrak c_{2m} = ||\xi ||^{-2m-2} \left ( \delta _{ab} -
				\frac{m}{3} R_{ajbk} x^{j} x^{k} \right ) \xi _{a} \xi _{b} + o(\mathbf{x^{2})},
				\\
				&\mathfrak c_{2m+1}=\frac{-2mi}{3} ||\xi ||^{-2m-2} \mathrm{Ric}_{ak} x^{k}
				\xi _{a} - 2 m i ||\xi ||^{-2m-2} \bigl( T_{ab} x^{b} \xi _{a} \bigr) +
				o(\mathbf{x)
				}\\
				&\mathfrak c_{2m+2}=\frac{m(m+1)}{3} ||\xi ||^{-2m-4} \mathrm{Ric}_{ab}
				\xi _{a}\xi _{b} - \frac{m}{4} R || \xi ||^{-2m-2} + o(\mathbf{1)\xch{,}{.}
			}\end{aligned}
			\label{symD-n}
		\end{equation}
		where
		\begin{equation*}
			T_{ab} = \frac{1}{8} R_{abjk} \gamma ^{j} \gamma ^{k},
		\end{equation*}
		and we have used the antisymmetry of $T_{ab}$. The Clifford images of 
		$v$ and $w$ have a local expansion around ${\mathbf{x}}=0$ in normal coordinates,
		\begin{equation*}
			\hat{v} = v_{a} \gamma ^{a} + o({\mathbf{1}}), \qquad \qquad \hat{w} = w_{a}
			\gamma ^{a} + w_{ab} \gamma ^{a} x^{b} + o({\mathbf{x}}),
		\end{equation*}
		where $w_{a}, v_{a}$ and $w_{ab}$ are constants. Note that here due to
		the metric in normal coordinates at $x=0$ being Euclidean, the gamma matrices
		satisfy $\{\gamma ^{a}, \gamma ^{b}\} = 2\delta ^{ab}$. Observe that the
		symbol of $\hat vD$ needs to be expanded to $o({\mathbf{1}})$ whereas symbol
		of $\hat wD$ to $o({\mathbf{x}})$, and so they read
		%
		\begin{equation}
			\begin{aligned}
				&\sigma (\hat vD) = i v_{a} \gamma ^{a} \gamma ^{j} i \xi _{j} +o({
					\mathbf{1}}),
				\\
				&\sigma (\hat wD) = i w_{a} \gamma ^{a} \gamma ^{j} \bigl( i \xi _{j} +
				\frac{1}{8} R_{\ell jps} x^{\ell }\gamma ^{p} \gamma ^{s} \bigr) - i w_{ab}
				\gamma ^{a} \gamma ^{j} \xi _{j} x^{b} + o({\mathbf{x}})\xch{.}{,}
			\end{aligned}
			\label{eq4.4}
		\end{equation}
		The constants $w_{ab}$ arise from the dependence of the one-form
		$w$ on the coordinates around $x=0$.
		
		The symbol of $\hat vD\hat wD$ up to order $o({\mathbf{1}})$ in normal coordinates
		is
		%
		\begin{equation}
			\begin{aligned}
				\sigma (\hat vD\hat wD) & = v_{a} w_{b} \gamma ^{a} \gamma ^{j}
				\gamma ^{b} \gamma ^{k} \xi _{j} \xi _{k} - \frac{1}{8} v_{a} w_{b}
				\gamma ^{a} \gamma ^{j} \gamma ^{b} \gamma ^{k} \gamma ^{p} \gamma ^{s}
				R_{jkps}
				\\
				& \quad -i v_{a} w_{bj} \gamma ^{a} \gamma ^{j} \gamma ^{b} \gamma ^{k}
				\xi ^{k} + o({\mathbf{1}}).
			\end{aligned}
			\label{eq4.5}
		\end{equation}
		We can omit additional terms with explicit dependence on normal coordinates
		as they will all vanish at $x=0$.
		
		The symbol of order $-n$ of the product $\hat vD \hat wD D^{-n}$ comes
		then (at $x=0$) as
		%
		\begin{equation}
			\begin{aligned}
				\sigma _{-n}(\hat vD \hat wD D^{-n}) & = \frac{m(m+1)}{3} v_{a} w_{b}
				\gamma ^{a} \gamma ^{j} \gamma ^{b} \gamma ^{k} \text{Ric}_{rs} \xi _{j}
				\xi _{k} \xi _{r} \xi _{s} ||\xi ||^{-2m-4}
				\\
				& - \frac{m}{4} v_{a} w_{b} \gamma ^{a} \gamma ^{j} \gamma ^{b}
				\gamma ^{k} R \xi _{j} \xi _{k} ||\xi ||^{-2m-2}
				\\
				& -\frac{2m}{3} v_{a} w_{b} \gamma ^{a} \gamma ^{j} \gamma ^{b}
				\gamma ^{k}\left ( \text{Ric}_{rj} \xi _{k} \xi _{r} + \text{Ric}_{rk}
				\xi _{r} \xi _{j} \right ) ||\xi ||^{-2m-2}
				\\
				& - \frac{m}{4} v_{a} w_{b} \gamma ^{a} \gamma ^{j} \gamma ^{b}
				\gamma ^{k} \gamma ^{p} \gamma ^{q} \left ( R_{rjpq} \xi _{r} \xi _{k}
				+ R_{rkpq} \xi _{r} \xi _{j} \right ) ||\xi ||^{-2m-2}
				\\
				& + \frac{m}{3} v_{a} w_{b} \gamma ^{a} \gamma ^{j} \gamma ^{b}
				\gamma ^{k} \left ( R_{rjpk} \xi _{r} \xi _{p} + R_{rkpj} \xi _{r}
				\xi _{p} \right ) ||\xi ||^{-2m-2}
				\\
				& -\frac{1}{8} v_{a} w_{b} \gamma ^{a} \gamma ^{j} \gamma ^{b}
				\gamma ^{k} \gamma ^{p} \gamma ^{s} R_{jkps} ||\xi ||^{-2m} +o({\mathbf{1}}).
			\end{aligned}
			\label{eq4.6}
		\end{equation}
		We see at once that it does not contain any term with $w_{ab}$ hence the
		result will be bilinear in the differential forms. Integrating over
		$\xi \in S^{n-1}$ and taking the trace over the matrices $\gamma $ we obtain:
		\begin{equation*}
			\mathcal{G}_{1}(v,w) = 2^{m} \frac{v_{n-1}}{6} \int _{M} \Bigl( G(v,w)
			+ \frac{n-2}{4} g(v,w) R \Bigr)~vol_{g}.
		\end{equation*}
		The functional $ \mathcal{G}_{2}$ requires the computation of the symbol
		of $D^{-n+2}$ up to $o({\mathbf{1}})$ as the product of two forms is an operator
		of order zero. Then, using the explicit formula for the symbol of
		$D^{-n+2}$ \reftext{\eqref{symD-n}}, we obtain
		\begin{equation*}
			\mathcal{G}_{2}(v,w) = - 2^{m} \frac{v_{n-1}}{6} \int _{M}
			\frac{n-2}{4} g(v,w) R~vol_{g}.
		\end{equation*}
		This shows our statement about
		$\mathcal{G}_{D}=\mathcal{G}_{1}+\mathcal{G}_{2}$.
	\end{proof}
	%
	
	\section{Towards noncommutative metric and Einstein functionals}
	\label{sec5}
	
	The spectral methods that we have proposed to obtain the metric and Einstein
	functionals are well suited for generalisation to the noncommutative case.
	So far, almost exclusively scalar geometric quantities (like scalar curvature)
	were computed for noncommutative tori, with conformally rescaled Laplace
	and Dirac operators (see \cite{FK19} for a review) and partial conformal
	rescaling \cite{DaSi15,CF19}. As the tensors carry more information than
	scalar geometric objects their study on noncommutative level is certainly
	more interesting though definitely complicated, especially that a general
	algebra may have no outer derivations which are usually regarded as the
	noncommutative counterpart of vector fields. On the other hand, differential
	forms are naturally associated to any spectral triple, with the latter
	deemed to best encode the notion of a Riemannian manifold in the noncommutative
	case.
	
	On the quantum tori we take advantage of both of these structures and correspondingly
	propose definitions of the associated spectral metric and Einstein tensors
	using the generalisation of the Wodzicki residue on the generalised algebra
	of symbols on noncommutative tori. We compute explicitly for 2- and 4-dimensional
	noncommutative tori the relevant functionals of outer derivations for the
	conformally rescaled Laplace operators and of one-forms with conformally
	rescaled Dirac operators. We also define analogous functionals for regular
	finitely summable spectral triples and study their product with the simplest
	nontrivial finite spectral triple.
	
	\subsection{The metric and Einstein functionals for the Laplacian on a noncommutative tori}
	\label{51}
	
	The quantum tori are prominent examples of noncommutative manifolds
	which admit noncommutative analogues of many classical geometrical objects.
	In particular, there are outer derivations that act on the smooth algebra
	${\mathcal A}= C^{\infty}(\mathbb{T}^{n}_{\theta})$ and that can be interpreted
	as noncommutative vector fields, even though in general they form only
	a complex vector space rather than an ${\mathcal A}$-bimodule. It is then
	straightforward to identify a noncommutative counterpart of the flat-metric
	Laplace operator. This can also be generalised to the case of conformally
	rescaled geometry, where the conformal factor is taken as a positive element
	of the algebra of the noncommutative torus.
	
	We can therefore exploit our definition of the metric and Einstein spectral
	functionals to define the corresponding functionals for the noncommutative
	tori and compute them explicitly. For that purpose we use the pseudodifferential
	calculus of symbols as defined in \cite{CoTr11} and then used and developed by many authors (see \cite{FK19} for a review). The rules of the algebra
	of symbols are almost identical to the usual ones, however, with the partial
	derivatives replaced by derivations and the symbols valued in the algebra
	of noncommutative torus.
	
	This algebra admits a natural generalisation of the Wodzicki residue, which
	is defined as the trace of the integral over the cosphere
	$||\xi ||=1$ of the symbol of order $-n$ (where $n$ is the dimension of
	the torus). The existence of such Wodzicki residue trace on the algebra
	of symbols over the noncommutative 2-torus was demonstrated and discussed
	in \cite{FW11,LNP16,Si14}. We denote it again by $\mathcal{W}$.
	
	However, for the conformally rescaled $\mathbb{T}^{n}_{\theta}$ we find
	it convenient to work with an enlarged algebra $\hat{{\mathcal A}}$, which
	is generated by ${\mathcal A}$ and its copy ${\mathcal A}^{o}$ commuting
	with ${\mathcal A}$. Indeed for the Laplace-type operators, the relevant
	Weyl factor in principle can be taken from $\hat{{\mathcal A}}$, though
	usually it assumed to be in ${\mathcal A}$ and we adhere to this convention.
	For the spectral triple with conformally rescaled Dirac operator, the conformal factor is instead assumed to be from ${\mathcal A}^{o}$ as only in such a case the associated differential one-forms (generated by
	${\mathcal A}$ and its commutators with the Dirac operator) are bounded operators, on which our functionals will be defined.
	
	Now, the standard flat Dirac operator on $\mathbb{T}^{n}_{\theta}$ as well
	as its conformal rescaling are clearly first-order differential operators
	for the extended pseudodifferential calculus with symbols valued in the
	algebra $\hat{{\mathcal A}}$, since they are just built from the derivations
	on ${\mathcal A}$ which extend to derivations on
	$\hat{{\mathcal A}}$. Next, considering for simplicity only the \textit{strictly
		irrational} \cite{Ri90} noncommutative torus (with the center of algebra
	equal to $\mathbb{C}$), there is an obvious factorized trace
	$\tau ^{\otimes}$ on the enlarged algebra $\hat{{\mathcal A}}$ given by
	$\tau ^{\otimes}(ab^{o})=\tau (a)\tau (b^{o})$. Moreover, since
	$\tau ^{\otimes}$ is invariant under derivations like $\tau $, we use it
	to define the tracial Wodzicki residue on $\hat{{\mathcal A}}$-valued symbols
	as above, and still denote it by $\mathcal{W}$ and its density by
	$\mathcal{w}$.
	
	\subsubsection{The metric and Einstein functionals for the conformally rescaled Laplacian on a noncommutative 2-torus}
	\label{511}

	Since every two-dimensional Riemannian manifold has a vanishing Einstein
	tensor, it is natural to ask whether this holds also for the noncommutative
	torus. We shall demonstrate that it is true in the case of a conformally
	rescaled Laplace operator. We denote by
	${\mathcal A}= C^{\infty}(\mathbb{T}^{2}_{\theta})$ the algebra of smooth
	elements of the noncommutative two-torus and by $\tau $ the standard trace
	over its $C^{\ast}$-algebraic completion. By
	${\mathcal H}= L^{2}(\mathbb{T}^{2}_{\theta}, \tau )$ we denote the standard
	Hilbert space obtained by the GNS construction for the tracial state
	$\tau $. By $\delta _{1}, \delta _{2}$ we denote the basis of derivations
	of the algebra ${\mathcal A}$ implemented as densely defined operators
	on ${\mathcal H}$.
	
	\begin{defn}
		\label{defn5.1}
		Let $h \in C^{\infty}(\mathbb{T}^{2}_{\theta})$ be positive, invertible,
		with a bounded inverse. We define as the conformally rescaled Laplace operator
		for the noncommutative torus the following densely defined selfadjoint
		operator on ${\mathcal H}$:
		%
		\begin{equation}
			\Delta _{h} = h^{-1} \Delta \, h^{-1} ,
			\label{eq5.1}
		\end{equation}
		where
		%
		\begin{equation}
			\Delta = \sum _{a=1,2} \delta _{a}^{2}.
			\label{eq5.2}
		\end{equation}
	\end{defn}
	This definition is motivated by the commutative case $\theta =0$ where
	the operator $\Delta _{h}$ is unitarily equivalent to the operator
	$h^{-2} \Delta $, which is a densely defined self-adjoint operator on
	${\mathcal H}_{h} = L^{2}(\mathbb{T}^{2}, \tau _{h})$, where
	$\tau _{h}(a) = \tau (h^{2} a)$.\eject
	
	Similarly, we define the vector fields as appropriate self-adjoint generalisations
	of operators unitarily equivalent to derivations,
	\begin{equation*}
		V_{h} = \sum _{a=1,2} V^{a} h \delta _{a} h^{-1},
	\end{equation*}
	where $V^a \in \mathbb{C}$. %
	\begin{prop}
		\label{prop5.2}
		For the conformally rescaled Laplace operator on a noncommutative 2-torus
		the metric functional reads
		\begin{equation*}
			\mathcal{g}^{\Delta _{h}}(V_{h},W_{h}) = \mathcal{W}\left ( V_{h} W_{h}
			\Delta _{h}^{-2} \right ) = \pi \tau (h^{4}) V^{a} W^{b} \delta _{ab},
		\end{equation*}
		whereas the spectral Einstein functional and its density vanish identically
		\begin{equation*}
			\mathcal{G}^{\Delta _{h}}(V_{h},W_{h}) = \mathcal{W}\left ( V_{h} W_{h}
			\Delta _{h}^{-1} \right ) = 0.
		\end{equation*}
	\end{prop}
	\begin{proof}
		The metric functional is straightforwardly computed
		%
		\begin{equation}
			\begin{aligned}
				\mathcal{g}^{\Delta _{h}}(V_{h},W_{h}) :&= \mathcal{W}\left ( (h V^{a}
				\delta _{a} h^{-1}) (h W^{b} \delta _{b} h^{-1}) ( h^{-1} \Delta h^{-1}
				)^{-2}\right )
				\\
				&= \mathcal{W}\left ( h V^{a} W^{b} \delta _{a} \delta _{b} \Delta ^{-1}
				h^{2} \Delta ^{-1} h \right )
				\\
				&= \pi \tau \left ( h^{4} V^{a} W^{b} \delta _{ab} \right ) = \pi
				\tau (h^{4}) V^{a} W^{b} \delta _{ab}.
			\end{aligned}
			\label{eq5.3}
		\end{equation}
		Note that the components of the noncommutative metric tensor functional
		scale as $h^{4}$ as in the classical situation.
		
		Next, we compute the Einstein functional,
		%
		\begin{equation}
			\begin{aligned}
				\mathcal{G}^{\Delta _{h}}(V_{h},W_{h}) :&= \mathcal{W}\left ( (h V^{a}
				\delta _{a} h^{-1}) (h W^{b} \delta _{b} h^{-1}) ( h^{-1} \Delta h^{-1}
				)^{-1}\right )
				\\
				& = \mathcal{W}\left ( h V^{a} W^{b} \delta _{a} \delta _{b}
				\Delta ^{-1} h \right ) = \mathcal{W}\left ( h^{2} V^{a} W^{b}
				\delta _{a} \delta _{b} \Delta ^{-1} \right ).
			\end{aligned}
			\label{eq5.4}
		\end{equation}
		Since $\Delta $ is the standard (flat) Laplace operator and $V,W$ are just
		$\mathbb{C}$-linear combinations of the derivations,
		$V^{a} W^{b} \delta _{a} \delta _{b} \Delta ^{-1}$ is an operator that
		has exclusively symbol of order $0$. Since its symbol of order $-2$ vanishes,
		its product with any element from the algebra, like $h^{2}$, has the same
		property, and thus the Wodzicki residue of it vanishes.
	\end{proof}
	Note that by using the same argument and the trace property of the Wodzicki
	residue we also show that for any ${\mathcal a} \in {\mathcal A}$ the localised
	functional vanishes
	\begin{equation*}
		{\mathcal G}^{\Delta _{h}}({\mathcal a} ,V,W) = \mathcal{W}\left ( {
			\mathcal a} V W \Delta ^{-1}_{h} \right ) = 0
	\end{equation*}
	(here ${\mathcal a} V$, in general, does not implement an algebra derivation).
	Thus, the conformally rescaled geometry of the noncommutative 2-torus has
	indeed the same property as a 2-dimensional manifold - its Einstein tensor
	vanishes.
	
	We finish this example by reiterating that the analysis does not change
	if $h$ is taken from the larger algebra $\hat{{\mathcal A}}$.
	
	\subsubsection{The metric and Einstein functionals for the conformally rescaled Laplacian on a noncommutative 4-torus}
	\label{512}

	As the analogue of the conformally rescaled Laplace operator over the noncommutative
	4-torus (acting on
	${\mathcal H}= L^{2}(\mathbb{T}^{2}_{\theta}, \tau )$) we simply take,
	%
	\begin{equation}
		\Delta _{h} = \sum _{a=1}^{4} \chi ^{-1}\cdot \delta _{a} \cdot \chi
		\cdot \delta _{a} \cdot \chi ^{-1},
		\label{eq5.5}
	\end{equation}
	where $\chi = h^{2}$ is a positive element of the algebra
	$C^{\infty}(T^{4}_{\theta})$ and $\delta _{a}$ are the standard derivations
	extended as operators on a dense subspace of
	$L^{2}(T^{4}_{\theta},\tau )$. The deformation parameter $\theta $ is here,
	in fact, a matrix (such that the center of the algebra is
	$\mathbb{C}$). Note that the operator $\Delta _{h}$ is again a noncommutative
	generalisation of an operator that classically (for $\theta =0$) is unitarily
	equivalent to a Laplace operator for the conformally rescaled metric on
	the four-torus.
	
	Similarly, as in the case of 2-dimensional noncommutative torus we take
	two arbitrary vector fields $V,W$ understood as $h$-rescaled linear combinations
	of derivations,
	\begin{equation*}
		V_{h} = V^{a} \chi \delta _{a} \chi ^{-1}, \qquad W_{h} = W^{b} \chi
		\delta _{b} \chi ^{-1}.
	\end{equation*}
	%
	\begin{prop}
		\label{prop5.3}
		The spectral metric and Einstein functionals on derivations for a conformally
		rescaled Laplace operator over the noncommutative 4-torus are, respectively,
		%
		\begin{equation}
			\begin{aligned}
				{\mathcal g}^{{\Delta _{h}}}(V_{h},W_{h}) = & \, 2\pi ^{2} \tau (
				\chi ^{3}) V^{a} W^{b} \delta _{ab},
				\\
				{\mathcal G}^{{\Delta _{h}}}(V_{h},W_{h}) = & \, 2\pi ^{2} \tau
				\biggl( -\frac{1}{24} \chi (V^{a} \delta _{a} \chi ) \chi ^{-1} (W^{b}
				\delta _{b} \chi ) -\frac{1}{24} \chi (W^{a} \delta _{a} \chi ) \chi ^{-1}
				(V^{b} \delta _{b} \chi )
				\\
				& +\frac{5}{24} (V^{a} \delta _{a} \chi ) \chi ^{-1} (W^{b} \delta _{b}
				\chi ) \chi +\frac{5}{24}(W^{a} \delta _{a} \chi ) \chi ^{-1} (V^{b}
				\delta _{b} \chi ) \chi
				\\
				& -\frac{1}{24} (V^{a} \delta _{a} \chi ) (W^{b} \delta _{b} \chi ) -
				\frac{1}{24} (W^{a} \delta _{a} \chi ) (V^{b} \delta _{b} \chi )
				\\
				& - \frac{1}{3} V^{a} W^{b} (\delta _{ab} \chi ) \chi + \frac{1}{6}
				\chi V^{a} W^{b} (\delta _{ab} \chi )
				\\
				&+ \biggl( - \frac{1}{24} (\delta _{a} \chi ) \chi ^{-1} (\delta _{a}
				\chi ) \chi - \frac{1}{24} \chi (\delta _{a} \chi ) \chi ^{-1} (
				\delta _{a} \chi )
				\\
				&\qquad - \frac{1}{24} (\delta _{a} \chi ) (\delta _{a} \chi ) +
				\frac{1}{12} \chi (\Delta \chi ) + \frac{1}{12} (\Delta \chi ) \chi
				\biggr) V^{b} W^{b} \biggr).
			\end{aligned}
			\label{eq5.6}
		\end{equation}
	\end{prop}
	\noindent
	We skip the proof, which is a straightforward and tedious computation of
	the respective symbols and the integration over the cosphere. Note that
	the above expression can be further simplified using the trace property
	of $\tau $. Let us also note that the expression has a well-defined commutative
	limit, in which case the spectral Einstein functional density becomes as
	follows:
	%
	\begin{equation}
		\begin{aligned}
			\mathcal{w}(VW{\Delta _{h}}^{-2}) = & 2\pi ^{2} \biggl(\frac{1}{4} V^{a}
			W^{b} (\delta _{a} \chi ) (\delta _{b} \chi ) - \frac{1}{6} V^{a} W^{b}
			(\delta _{ab} \chi ) \chi
			\\
			&- \frac{1}{8} (\delta _{a} \chi ) (\delta _{a} \chi ) (V^{b} W^{b}) +
			\frac{1}{6} (\Delta \chi ) \chi (V^{b} W^{b}) \biggr)\xch{,}{.}
		\end{aligned}
		\label{eq5.7}
	\end{equation}
	and is exactly equal to:
	\begin{equation*}
		\frac{2 \pi ^{2}}{6} \sqrt{g} \, G_{ab} V^{a} W^{b}.
	\end{equation*}
	
	\subsection{Spectral Einstein and metric tensor for spectral triples}
	\label{52}
	
	Let $({\mathcal A}, D, {\mathcal H})$ be a $n$-summable spectral triple and
	$\Omega ^{1}_{D}({\mathcal A})$ be the ${\mathcal A}$ bimodule of one-forms generated
	by ${\mathcal A}$ and $[D, {\mathcal A}]$, which by definition consists
	of bounded operators on ${\mathcal H}$. We assume that there exists a generalised
	algebra of pseudodifferential operators which contains
	${\mathcal A}$, $D$, $|D|^{\ell}$ for $\ell \in \mathbb{Z}$, and there
	exists a tracial state $\mathcal{W}$ on it, still called a noncommutative
	residue, which identically vanishes on $T |D|^{-k}$ for any $k>n$ and a
	zero-order operator $T$ (an operator in the algebra generated by
	${\mathcal A}$ and $\Omega^{1}_D({\mathcal A})$). We propose the following
	definition of metric and Einstein functionals for spectral triples.
	%
	\begin{defn}%
		\label{Dwres}
		The spectral metric functional on differential forms is
		\begin{equation*}
			\mathcal{g}_{D}(v,w) = \mathcal{W}(v w|D|^{-n}),
		\end{equation*}
		where $v,w \in \Omega ^{1}_{D}({\mathcal A})$, and the Einstein functional
		is
		\begin{equation*}
			\mathcal{G}_{D}(v,w) = \mathcal{W}(v \{ D, w \} D |D|^{-n}).
		\end{equation*}
	\end{defn}
	Note that since $\Omega ^{1}_{D}({\mathcal A})$ is a bimodule over
	${\mathcal A}$ these functionals are already localised, as one can replace
	$v$ by $av$. A simple consequence of the properties of $\mathcal{W}$ on the generalised pseudodifferential calculus for regular spectral triples
	are the following linearities for the spectral metric functional
	over one-forms,
	\begin{equation*}
		{\mathcal g}_{D} (vb,w) = {\mathcal g}_{D} (v, bw), \qquad {
			\mathcal g}_{D} (av,w) = {\mathcal g}_{D} (v, wa).
	\end{equation*}
	However, unlike in the classical case it is not entirely obvious whether
	any form of linearity over ${\mathcal A}$ in the second entry for the Einstein functional will hold in general. To ensure this property, we propose the following requirement on a class of spectral triples.
	
	\begin{defn}
		\label{defn5.5}
		We say that a spectral triple with a trace on the generalised algebra of
		pseudodifferential operators is spectrally closed if for any zero-order
		operator $T$ (as defined earlier) the following holds:
		\begin{equation*}
			\mathcal{W}(T D |D|^{-n}) = 0.
		\end{equation*}
	\end{defn}
	%
	\begin{lem}
		\label{lem5.6}
		The classical spectral triple over a closed oriented spin-c manifold
		$M$ of dimension $n=2m$ is spectrally closed in the above sense.
	\end{lem}
	\begin{proof}
		Computing the symbol of $D |D|^{-n} = D D^{-2m}$ using the expansion \reftext{\eqref{symD-n}} we see that it vanishes identically at ${\mathbf{x}}=0$. Since
		$T$ is an operator of order zero (does not depend on $\xi $) the product
		$T D D^{-2m}$ also has a vanishing symbol of order $-n$, hence both the
		density of the Wodzicki residue and the residue itself vanish identically.
	\end{proof}
	%
	\begin{thm}
		\label{thm5.7}
		If a spectral triple with a noncommutative residue over the generalised
		pseudodifferential calculus is spectrally closed, then the localised Einstein
		functional satisfies
		\begin{equation*}
			\mathcal{G}_{D}(v{\mathcal{b}} ,w) = \mathcal{G}_{D}(v,{\mathcal b}w),
		\end{equation*}
		for ${\mathcal b} \in {\mathcal A}$.
	\end{thm}

	\begin{proof}
		We compute
		\begin{equation*}
			\begin{aligned}
				\mathcal{G}_{D}(v, {\mathcal b}w) &= \mathcal{W}\, \bigl(v \{ D, {
					\mathcal b} w \} D |D|^{-n}\bigr)
				\\
				& = \mathcal{W}\, \bigl(v ([D, {\mathcal b}] w + {\mathcal b}\{ D,w
				\} )D |D|^{-n})\bigr)
				\\
				&= \mathcal{W}\, \bigl(v {\mathcal b} \{D ,w \} D |D|^{-n}\bigr) =
				\mathcal{G}_{D}(v{\mathcal b}, w),
			\end{aligned}
		\end{equation*}
		since the first term in the second line contains a product of three one-forms
		and therefore is a 0-th order term, so by the assumption of spectral closedness
		it vanishes. \end{proof}
	%
	\begin{rem}
		\label{rem5.8}
		Note that in the case of spectral triple over a manifold, since one-forms
		commute with ${\mathcal a},{\mathcal b} \in {\mathcal A}$, the Einstein functional has the same linearities over the algebra as a tensor.
	\end{rem}
	
	Next, we present two important situations where the assumptions before
	\reftext{Definition~\ref{Dwres}} are satisfied. The first is based on the algebra
	of pseudodifferential operators over noncommutative tori
	\cite{Co80,CoTr11} and the second on the abstract algebra of pseudodifferential
	operators for regular $n$-summable spectral triples \cite{CoMo95}.
	
	\subsubsection{The metric and Einstein functionals for the conformally rescaled spectral triple on a noncommutative tori}
	\label{sec5.2.1}

	We begin with the spectral triple on the noncommutative $n$-torus,
	$({\mathcal A}, {\mathcal H}, D_{k})$, where ${\mathcal A}$ is the
	$\delta _{a}$-smooth subalgebra and $D_{k}:=kDk$ is the conformal rescaling
	of the standard (flat) Dirac operator
	$D=\sum _{a} \gamma _{a} \delta _{a}$ with $\gamma _{a}$ as in Section~\ref{13}.
	As argued in \cite{DaSi15} the conformal factor $k>0$ has to be taken from
	the copy ${\mathcal A}^{o}$ of ${\mathcal A}$ in its commutant. The bimodule of one-forms, generated by the commutators $[D_{k},{\mathcal a}]$, ${\mathcal a} \in {\mathcal A}$, is a free left module generated by $ k^{2} \gamma ^{j}$. With assumptions as in Section~\ref{51}
	using the extension of the calculus in \cite{Co80} to
	$\hat{{\mathcal A}}$-valued symbols and of the analogue of the Wodzicki
	residue, the calculations are very much similar to the case of functionals
	for the Laplace operator on noncommutative tori. We provide them explicitly
	for the two lowest even dimensional cases.
	
	\begin{Example} \label{exmp1} Functionals of the conformally rescaled spectral triple
		on the noncommutative 2-torus.  
		
		As argued in \cite{DaSi15} it is possible to construct a usual (untwisted)
		spectral triple with a conformally rescaled Dirac operator, however, the
		conformal factor $k>0$ has to be taken from the commutant of the algebra.
		With ${\mathcal A}= C^{\infty}(\mathbb{T}^{2}_{\theta})$ and
		${\mathcal A}^{o}$ a copy of ${\mathcal A}$ in the commutant of
		${\mathcal A}$, such a spectral triple is given by
		$ ({\mathcal A}, D_{k} = k D k, {\mathcal H}\otimes \mathbb{C}^{2})$, where
		$D = D_{1}$ and $k \in {\mathcal A}^{o}$ is the standard flat Dirac operator,
		\begin{equation*}
			D = \left (
			\begin{matrix}
				0 & \delta _{2} - i \delta _{1}
				\\
				\delta _{2} + i \delta _{1} & 0
			\end{matrix}
			\right ).
		\end{equation*}
		In fact $D_{k}$ is an analogue of the classical Dirac operator for the
		flat metric rescaled conformally and unitarily transformed to act on Hilbert
		space with the volume measure of the flat metric.
		
		The space of one-forms is a bimodule over ${\mathcal A}$ generated by all
		commutators $[D_{k}, {\mathcal a} ]$,
		${\mathcal a} \in {\mathcal A}$ and it can be shown that it is a free left
		module generated by $ k^{2} \sigma ^{j}$, $j=1,2$, where
		$\sigma ^{j}$ are Pauli matrices.
		
		First, we show,
		%
		\begin{lem}
			\label{lem5.9}
			The conformally rescaled spectral triple on the noncommutative 2-torus,
			$ (\mathcal{A}, D_k = k D k, \mathcal{H} \otimes \mathbb{C}^2)$,
			is spectrally closed.
		\end{lem}
		\begin{proof}
			We have
			\begin{equation*}
				\sigma _{-2}(D_{k}^{-1}) = \sigma _{-2}(k^{-1} D^{-1} k^{-1}) =
				\sigma ^{p} (\frac{\delta _{pq}}{||\xi ||^{2}} - 2
				\frac{\xi _{p} \xi _{q}}{||\xi ||^{4}} \bigr) (-k^{-2} \delta _{q} k
				\, k^{-1}).
			\end{equation*}
			Integrating it over the $||\xi ||=1$ cosphere gives identically zero, therefore
			the Wodzicki residue of any expression that is a product of a zero-order
			operator with $D_{k}^{-1}$ vanishes.
		\end{proof}
		Next, we have the following.
		%
		\begin{prop}
			\label{prop5.10}
			For the conformally rescaled spectral triple over the noncommutative 2-torus
			the metric functional for $ v = k^{2} V^{a} \sigma ^{a}$ and
			$w = k^{2} W^{a} \sigma ^{a}$, $V^{a}, W^{a} \in {\mathcal A}$, reads
			\begin{equation*}
				{\mathcal g}_{D_{k}}(v,w) = \tau (V^{a} W^{a}),
			\end{equation*}\goodbreak\noindent
			whereas the spectral Einstein functional vanishes identically,
			\begin{equation*}
				{\mathcal G}_{D_{k}}(v,w) =0.
			\end{equation*}
		\end{prop}
		\begin{proof}
			The computation of the metric functional is straightforward. Next, since
			the spectral triple is spectrally closed, it suffices to compute the spectral
			Einstein functional over the one-forms that generate the bimodule
			$\Omega _{D_{k}}({\mathcal A})$,
			\begin{equation*}
				\mathcal{G}_{j\ell} = \mathcal{W}\bigl( k^{2} \sigma ^{j} \{ D_{k}, k^{2}
				\sigma ^{\ell }\} D_{k} D_{k}^{-2} \bigr).
			\end{equation*}
			The expression in curly brackets can be rewritten as
			\begin{equation*}
				\sigma ^{j} k^{2} [D_{k}, k^{2} \sigma ^{\ell }] k^{-1} D^{-1} k^{-1} +
				2 \sigma ^{j} \sigma ^{\ell }k^{4}.
			\end{equation*}
			Since the second term is of order zero, we only need to compute the symbol
			of order $-2$ of the first term, which can be rewritten as
			\begin{equation*}
				\sigma ^{j} k^{3} [D, k^{2} \sigma ^{\ell }] D^{-1} k^{-1} .
			\end{equation*}
			Then, using the tracial property of $\mathcal{W}$ we get
			\begin{equation*}
				\mathcal{G}_{j\ell} = \mathcal{W}\bigl( \sigma ^{j} k^{2} [D, k^{2}
				\sigma ^{\ell }] D^{-1} \bigr)= \mathcal{W}\bigl( \sigma ^{j} k^{2} [D,
				k^{2}] \sigma ^{\ell }D^{-1} \bigr) + \mathcal{W}\bigl( \sigma ^{j} k^{4}
				[D, \sigma ^{\ell }] D^{-1} \bigr) ,
			\end{equation*}
			which vanishes since both \xch{expressions have a}{expressions has a} vanishing symbol of order
			$-2$.
		\end{proof}
		It is rewarding to see that the conformally rescaled spectral triple of
		the noncommutative 2-torus shares indeed the same property with its commutative
		counterpart, namely, it has vanishing Einstein tensor.
	\end{Example}
	\begin{Example}\label{exmp2} The functionals for the conformally rescaled spectral triple
		on the noncommutative 4-torus.  

		The situation is more complicated for the strictly irrational 4-dimensional
		non-commutative torus. Following the classical situation, given a positive
		invertible $k \in {\mathcal A}^{o}$ (which is in the commutant of
		${\mathcal A}$), we take $D_{k} = k D k$ as the conformally rescaled and
		unitarily transformed to act on
		$L^{2}(\mathbb{T}^{4}_{\theta}, \tau ) \otimes \mathbb{C}^{4}$, operator
		$D$. Let $v$ and $w$ be two arbitrary one-forms in the Clifford algebra
		generated by ${\mathcal A}$ and $[D, {\mathcal A}]$,
		\begin{equation*}
			v = k^{2} V^{a} \gamma ^{a}, \qquad W = k^{2} W^{a} \gamma ^{a},
		\end{equation*}
		where $\gamma ^{a}$, $a=1,\ldots,4$ are the standard gamma matrices,
		$\{ \gamma ^{a}, \gamma ^{b} \} = 2 \delta _{ab}$, and $V^{a},W^{a}$ are
		elements of ${\mathcal A}$ (which, of course, commute with $k$). First,
		we establish.
		%
		\begin{lem}
			\label{lem5.11}
			The conformally rescaled spectral triple for a 4-dimensional noncommutative
			torus is spectrally closed.
		\end{lem}
		\begin{proof}
			The computations show that the integral over the cosphere of the symbol
			of $D_{k}^{-3}$ of order $-4$ is:
			%
			\begin{equation}
				\int \limits _{||\xi ||=1}\sigma _{-4}(D_{k}^{-3}) d\xi = \frac{i}{2}
				\gamma ^{a} \biggl( k^{-2} \bigl[ k^{-2} , \{k^{-1}, \delta _{a}k \}
				\bigr] k^{-2} \biggr).
				\label{eq5.8}
			\end{equation}
			It is easy to see that the trace of the above expression vanishes also
			if multiplied by any element of zero-order. This is so because any element
			of the algebra of operators of zero order $T$, contains only products of
			$\gamma $, powers of $k$ and elements from ${\mathcal A}$, and thus it
			commutes with $k$. Therefore taking the trace over the algebra gives
			\begin{equation*}
				\mathcal{W}\, (T D_{k}^{-3}) = 0.
				\qedhere \end{equation*}
		\end{proof}
		Next, we explicitly compute the metric and Einstein tensor functional for
		$v,w$.
		%
		\begin{prop}
			\label{prop5.12}
			The metric and the Einstein functionals for the conformally rescaled spectral
			triple on the noncommutative 4-torus are, respectively,
			%
			\begin{equation}
				\label{4}
				\begin{aligned}
					\mathcal{g}_{D_{k}}(v,w) =& \, \tau \left ( W^{a} V^{b} k^{-4}
					\right ),
					\\
					\mathcal{G}_{D_{k}}(v,w) =& \, \tau \left ( V^{a} W^{b} \biggl(
					\frac{1}{3} k^{-4} (\delta _{a} k) k^{2} (\delta _{b} k) +
					\frac{2}{3} k^{-3} (\delta _{a} k) k^{1} (\delta _{b} k) + k^{-2} (
					\delta _{a} k) (\delta _{b} k) \right .
					\\
					& + \frac{2}{3} k^{-1} (\delta _{a} k) k^{-1} (\delta _{b} k) -
					\frac{4}{3} k (\delta _{a} k) k^{-3} (\delta _{b} k) - \frac{2}{3} k^{2}
					(\delta _{a} k) k^{-4} (\delta _{b} k)
					\\
					& + \frac{2}{3} k^{-1} (\delta _{a}\delta _{b} k) + \delta _{ab}
					\left ( \frac{1}{3} k^{-1} (\delta _{c} k) k^{-1} (\delta _{c} k) +
					\frac{1}{3} k^{2} (\delta _{c} k) k^{-4} (\delta _{c} k) \right .
					\\
					& \left . + \frac{2}{3} k^{1} (\delta _{c} k) k^{-3} (\delta _{c} k) -
					\frac{2}{3} k^{-1} (\Delta k) \right ) \biggr).
				\end{aligned}
				%
			\end{equation}
		\end{prop}
		We skip the proof which consists of a straightforward though tedious computation
		noting only that to get \reftext{\eqref{4}} we used the cyclicity of the trace.
		
		\begin{rem}
			\label{rem5.13}
			The above expression has a well-defined commutative limit,
			%
			\begin{equation}
				\begin{aligned}
					\mathcal{G}_{D_{k}}(v,w) =& \frac{1}{\text{vol}(\mathbb{T}^{4})}
					\int \limits _{\mathbb{T}^{4}} V^{a} W^{b} \left ( \left (
					\frac{2}{3} k^{-2} k_{a} k_{b} + \frac{2}{3} k^{-1} k_{ab} \right )
					\right .
					\\
					& \left . \qquad + \left ( \frac{4}{3} \delta _{ab} k^{-2 } k_{c} k_{c}
					- \frac{2}{3} \delta _{ab} k^{-1 } k_{cc} \right ) \right )\xch{,}{.}
				\end{aligned}
				\label{eq5.10}
			\end{equation}
			which recovers the classical formula for the Einstein tensor for a conformally
			deformed metric on the 4-torus with the metric
			$g_{ab}= k^{-4} \delta _{ab}$,
			%
			\begin{equation}
				G_{ab} = 4 \left ( k^{-2} k_{a} k_{b} + k^{-1} k_{ab} \right ) + 8
				\delta _{ab} k^{-2 } k_{c} k_{c} - 4 \delta _{ab} k^{-1 } k_{cc}.
				\label{eq5.11}
			\end{equation}
			Note that the density of the functional involving the one-forms uses the
			contravariant Einstein tensor and the volume form, and scales with
			$k$ like $ k^{4} k^{4} \sqrt{k^{-16}}$ so the density in 4 dimensions scales
			exactly like the covariant Einstein tensor. The scaling factor comes from
			the convention that $\tau $ is normalised with $\tau (1)=1$.
		\end{rem}
	\end{Example}
	\subsubsection{The metric and Einstein functionals for a regular finitely summable spectral triple}
	\label{sec5.2.2}
	
	The second situation, which is different from the conformally rescaled
	geometries of noncommutative tori, is the case of regular, finitely summable
	spectral triples (for simplicity with simple dimension spectrum). In this
	case there is a PDO algebra and calculus of symbols as defined by Connes
	and Moscovici \cite{CoMo95} (see also \cite{Hi06} and \cite{Uuye11}) and
	there exists a tracial state (see \cite{CoMo95} Proposition II.1) denoted
	$\mathcal{W}$ (as before). The definition and methods of computation
	of the spectral metric and Einstein functionals are, however, quite involved
	and explicit expressions are feasible only in some special cases of highly
	symmetric Dirac operators (like the flat Dirac operator on noncommutative
	tori or fully symmetric Dirac operators on spheres). One can study though
	some functorial properties, for example the behaviour of the functionals
	under tensor product of regular spectral triples, of which a particular
	instance is a finite spectral triple as the second factor. When the first
	factor is the classical spectral triple (corresponding to spin-c manifolds)
	this corresponds to \textit{almost commutative geometries} which play a significant
	role as models for the geometry of physical interactions.
	
	Below we present an example of the metric and Einstein functionals for
	the tensor product of a regular (not necessarily commutative) spectral
	triple with a simplest non-trivial spectral triple on two points, expressing
	them in terms of the functionals on the components.
	
	\begin{Example}
		Functionals for the spectral triple on
		${\mathcal A}\otimes \mathbb{C}^{2}$.
	\end{Example}
	
	We assume that $({\mathcal A}, D, {\mathcal H})$ is an even spectral triple
	with grading $\gamma $ of dimension $n$ satisfying the assumptions described
	in section \ref{52}, which is spectrally closed, and we consider
	$({\mathcal A}\otimes \mathbb{C}^{2}, {\mathcal D}, {\mathcal H}
	\otimes \mathbb{C}^{2})$, where
	\begin{equation*}
		{\mathcal D} = \left (
		\begin{array}{c@{\quad}c}
			D & \gamma c
			\\
			\gamma c^{*} & D
		\end{array}
		\right ),
	\end{equation*}
	where $c \in \mathbb{C}$.
	
	It is easy to see that the bimodule of one-forms, associated to
	${\mathcal D}$ consists of the following operators,
	\begin{equation*}
		{\omega} = \left (
		\begin{array}{c@{\quad}c}
			w_{+} & \gamma c \phi _{+}
			\\
			\gamma c^{*} \phi _{-} & w_{-}
		\end{array}
		\right ),
	\end{equation*}
	where $w_{\pm }\in \Omega ^{1}_{D}({\mathcal A})$ and
	$ \phi _{\pm }\in {\mathcal A}$.
	
	We can now state
	%
	\begin{prop}
		\label{prop5.14}
		Let $\omega , \omega '$ be two one-forms in the spectral triple
		$({\mathcal A}\otimes \mathbb{C}^{2}, {\mathcal D}, {\mathcal H}
		\otimes \mathbb{C}^{2})$. Then,
		\begin{equation*}
			{\mathcal g}_{\mathcal D}(\omega , \omega ') = {\mathcal g}_{D}(w_{+},w_{+}')
			+ {\mathcal g}_{D}(w_{-},w_{-}') + cc^{*} {\mathcal v} \left (\phi _{+}
			\phi _{-}' + \phi _{-} \phi _{+}'\right ),
		\end{equation*}
		and
		\begin{equation*}
			\begin{aligned}
				{\mathcal G}_{\mathcal D}(\omega , \omega ') &= {\mathcal G}_{D}(w_{+},w_{+}')
				+ {\mathcal G}_{D}(w_{-},w_{-}'),
				\\
				& \; +cc^{*} {\mathcal g}_{D}(w_{+}-w_{-},w_{+}'-w_{-}') -
				\frac{n}{2} cc^{*} \bigl( {\mathcal g}_{D}(w_{+},w_{+}') + {
					\mathcal g}_{D}(w_{-},w_{-}') \bigr)
				\\
				& \; +cc^{*} \left ( {\mathcal g}_{D}(w_{+}, d\phi _{+}') - {
					\mathcal g}_{D}(d\phi _{+}, w_{-}') +{\mathcal g}_{D}(w_{-}, d\phi _{-}')
				- {\mathcal g}_{D}(d\phi _{-}, w_{+}') \right )
				\\
				& + (cc^{*})^{2} {\mathcal v}\bigl( (\phi _{+} + \phi _{-})(\phi _{+}'
				+ \phi _{-}') \bigr),
			\end{aligned}
		\end{equation*}
		where $d\phi =[D,\phi ]$  for $ \phi \in {\mathcal A}$.
	\end{prop}
	We skip the computational proof mentioning only that the condition of being
	spectrally closed does not need to be preserved in the tensor product with
	a finite spectral triple.
	
	\section{Final remarks}
	\label{sec6}
	
	The concept that geometric objects like tensors (metric, torsion and curvature
	tensors) can be expressed using spectral methods provides an invaluable
	possibility to study them globally both for the manifolds as well as for
	various extensions of geometries like noncommutative geometry.
	
	First of all, the metric functionals we constructed should be compared
	with various other concepts for the metric tensor proposed in the noncommutative
	realm on the algebraic level. Furthermore, possible relations with the
	notion of a distance (between states) and quantum metric spaces should
	be examined. The possibility of studying spectral functionals linked to
	connections leads to the possibility of defining an abstract notion of
	torsion and torsion-free connection. This would provide a natural contact
	with the various concepts of linear connections and possibly then with
	such notions as the Levi-Civita connection in the noncommutative case,
	for example \cite{fgr99,Ro13,BM20,BGJ21a,BGJ21b}.
	
	The newly introduced Einstein functional can be further computed for a
	variety of spectral triples on interesting algebras. This includes, in
	particular, more general spectral triples for the noncommutative two-tori
	(cf. \cite{DaSi13}) for which we expect the spectral Einstein functional to vanish. This can be further generalised as
	%
	\begin{conj}
		A regular spectral triple of dimension 2 has a vanishing Einstein
		spectral functional.
	\end{conj}
	Even if the regularity assumption should be suitably supplemented
	such a result will show the robustness of the noncommutative generalisation of manifolds.
	
	Another direction is to follow a spectral definition of a noncommutative
	Einstein manifold (or, being more precise, an Einstein spectral triple).
	%
	\begin{defn}
		\label{defn6.2}
		A spectral triple is called an Einstein spectral triple if the spectral
		Einstein functional is proportional to the metric functional.
	\end{defn}
	The study of these objects in both the context of almost commutative geometries
	and applications in mathematical physics can be very interesting. We also
	hope that the proposed functionals may allow for a global spectral view
	on the Ricci and curvature tensors themselves and rephrase in the spectral
	language the notion of flat manifolds. This, together with the Einstein
	tensor playing a significant role in physics, can prove helpful in applications
	to quantum field theory and various incarnations of quantum gravity. A
	possibility of having a direct approach not only to the action functional
	but also to a version of equations of motion is another appealing direction
	of studies that could possibly be linked to a more general variational
	formula.
	
	Yet another direction to investigate might be connected with orbifolds
	and other singular generalisations of a Riemannian manifold, to see for
	example how our spectral functionals apprehend the singularities. Also
	their extension to the case of a twisted algebra of pseudodifferential
	operators as in \cite{MaYu17} is an appealing task.
	
	Another interesting problem remains the equivalence of the algebra of pseudodifferential
	operators (\cite{Co80} and \cite{CoMo95}) for the noncommutative tori with
	various Dirac operators. Comparison of the functionals discussed in this
	work can shed some light on this open problem.
	
	
	
	%
	
	%
	
	
		\appendix \section{Algebra of symbols of pseudodifferential operators}
		\label{appA}
		
		Suppose that $P$ and $Q$ are two pseudodifferential operators with symbols,
		%
		\begin{equation}
			\sigma (P)(x,\xi )=\sum \limits _{\alpha} \sigma (P)_{\alpha}(x)\xi ^{
				\alpha},
			\hspace{20pt}
			\sigma (Q)(x,\xi )=\sum \limits _{\beta} \sigma (Q)_{\beta}(x)\xi ^{
				\beta},
			\label{eqA.1}
		\end{equation}
		respectively, where $\alpha , \beta $ are multiindices. The composition
		rule for the symbols of their product takes the form \cite{Gi84}.
		%
		\begin{equation}
			\sigma (PQ)(x,\xi )=\sum \limits _{\beta}
			\frac{(-i)^{|\beta |}}{|\beta |!}\partial ^{\xi}_{\beta }\sigma (P)(x,
			\xi )\partial _{\beta }\sigma (Q)(x,\xi ),
			\label{composition}
		\end{equation}
		where $\partial _{a}^{\xi}$ denotes the partial derivative with respect
		to the coordinate of the cotangent bundle.
		
		We start with computation of the three leading coefficients of the symbols
		of $P^{-1}$ (we assume that the kernel of $P$ is finite dimensional and
		can be neglected in the following) for a second-order differential operator
		$P$, with the symbol expansion,
		%
		\begin{equation}
			\sigma (P)(x,\xi )=\mathfrak a_{2}+\mathfrak a_{1}+\mathfrak a_{0}.
			\label{eqA.3}
		\end{equation}
		The inverse is a pseudodifferential operator $P^{-1}$, with a symbol of
		the form,
		%
		\begin{equation}
			\sigma (P^{-1})(x,\xi )=\mathfrak b_{2}+\mathfrak b_{3}+\mathfrak b_{4}+...,
			\label{eqA.4}
		\end{equation}
		where $\mathfrak b_{k}$ is homogeneous in $\xi $ of order $-k$. Inserting
		these expressions into \reftext{\eqref{composition}} and taking homogeneous parts
		of order $0,-1$ and $-2$ we get the following set of equations:
		%
		\begin{equation}
			\begin{aligned}
				&\mathfrak a_{2} \mathfrak b_{2}=1,
				\\
				&\mathfrak a_{1} \mathfrak b_{2}+\mathfrak a_{2} \mathfrak b_{3} -i
				\partial _{a}^{\xi}(\mathfrak a_{2})\partial _{a}(\mathfrak b_{2})=0,
				\\
				&\mathfrak a_{2}\mathfrak b_{4}+\mathfrak a_{1} \mathfrak b_{3} +
				\mathfrak a_{0} \mathfrak b_{2} -i\partial _{a}^{\xi}(\mathfrak a_{1})
				\partial _{a}(\mathfrak b_{2})
				\\
				& \qquad \qquad -i\partial _{a}^{\xi}(\mathfrak a_{2}) \partial _{a}(
				\mathfrak b_{3}) -\frac{1}{2}\partial _{a}^{\xi}\partial _{b}^{\xi}(
				\mathfrak a_{2}) \partial _{a}\partial _{b}(\mathfrak b_{2})=0,
			\end{aligned}
			\label{eqA.5}
		\end{equation}
		which we solve recursively, obtaining
		%
		\begin{equation}
			\label{SymPinverse}
			\begin{aligned}
				&\mathfrak b_{2}=\mathfrak a_{2}^{-1},
				\\
				& \mathfrak b_{3}= -\mathfrak b_{2} \left ( \mathfrak a_{1}
				\mathfrak b_{2} -i\partial _{a}^{\xi}(\mathfrak a_{2})\partial _{a}(
				\mathfrak b_{2}) \right ),
				\\
				&\mathfrak b_{4}= -\mathfrak b_{2} \left ( \mathfrak a_{1}\mathfrak b_{3}
				+ \mathfrak a_{0} \mathfrak b_{2} - i\partial _{a}^{\xi}(\mathfrak a_{1})
				\partial _{a}(\mathfrak b_{2}) \right .
				\\
				& \qquad \qquad \left . -i\partial _{a}^{\xi}(\mathfrak a_{2})
				\partial _{a}(\mathfrak b_{3}) -\frac{1}{2}\partial _{a}^{\xi}
				\partial _{b}^{\xi}(\mathfrak a_{2})\partial _{a}\partial _{b}(
				\mathfrak b_{2}) \right ).
			\end{aligned}
			%
		\end{equation}
		
		Next, we show a technical lemma that allows us to compute the symbol of
		the higher inverse power of the Laplace operator. This is a straightforward
		application of the iterated formula \reftext{\eqref{composition}} and we include
		it only for completeness. To shorten the notation, we shall denote the
		derivatives with respect to the coordinates on the manifold by
		$\delta $ and use still $\partial $ for the partial derivatives with respect
		to the coordinates on the cotangent bundle. Further, we shorten the notation
		of the derivatives applied to a certain element of the product, which would
		be valid for all derivatives (both $\partial , \delta $), for example:
		\begin{equation*}
			\partial _{a}^{[k]} \left ( y_{1} y_{2} \cdots y_{n} \right ) = y_{1} y_{2}
			\cdots \partial _{a}(y_{k}) \cdots y_{n}.
		\end{equation*}
		
		\def\thethm{A.\arabic{thm}}
		\setcounter{thm}{0}
		\begin{lem}
			\label{LA1}
			Given the first three leading symbols of an operator $P$ (with the principal
			symbol being of order $-k$),
			\begin{equation*}
				\sigma (P) = \mathfrak p_{k} + \mathfrak p_{k+1}+ \mathfrak p_{k+2}+
				\cdots ,
			\end{equation*}
			we can express the first three leading symbols of $R=P^{l}$,
			\begin{equation*}
				\sigma (R) = \mathfrak r_{lk} + \mathfrak r_{lk+1} + \mathfrak r_{lk+2}
				+ \cdots ,
			\end{equation*}
			in terms of $\mathfrak p_{k}$, $\mathfrak p_{k+1}$,
			$\mathfrak p_{k+2}$, as:
			%
			\begin{equation}
				\begin{aligned}
					&\mathfrak r_{lk} = ({\mathfrak p}_k )^{l},
					\\
					& \mathfrak r_{lk+1} = \sum \limits _{j = 1}^{l} (\mathfrak p_{k})^{j-1}
					\mathfrak p_{k+1} (\mathfrak p_{k})^{l-j} -i \sum \limits _{1 \leq j <
						w \leq l} {\partial}_{a}^{[j]} {\delta}_{a}^{[w]} (\mathfrak p_{k})^{l},
					\\
					& \mathfrak r_{lk+2} = \sum \limits _{j = 1}^{l} (\mathfrak p_{k})^{j-1}
					\mathfrak p_{k+2} (\mathfrak p_{k})^{l-j}
					\\
					& \qquad \qquad + \sum \limits _{1 \leq j < s \leq l} (\mathfrak p_{k})^{j-1}
					\mathfrak p_{k+1} (\mathfrak p_{k})^{s-j-1} \mathfrak p_{k+1} (
					\mathfrak x)^{l-s},
					\\
					& \qquad \qquad - i \sum \limits _{s=1}^{l} \sum \limits _{1 \leq j < p
						\leq l} {\partial}_{a}^{[j]} {\delta}_{a}^{[p]} (\mathfrak p_{k})^{s-1}
					\mathfrak p_{k+1} (\mathfrak p_{k})^{l-s} ,
					\\
					& \qquad \qquad -\frac{1}{2}\sum \limits _{1 \leq j < p \leq l} \,\,
					\sum \limits _{1 \leq r < s \leq l} {\partial}_{b}^{[r]} {\delta}_{b}^{[s]}
					{\partial}_{a}^{[j]} {\delta}_{a}^{[p]} (\mathfrak p_{k})^{l},
				\end{aligned}
				\label{operator_powers}
			\end{equation}
		\end{lem}
		\begin{proof}
			We proceed by induction in $l$. The claim is obvious for $l=1$. Assume
			now that it holds for $l-1$, and consider $R = P^{l}=P P^{l-1}$. By \reftext{\eqref{composition}} we can see, that
			%
			\begin{equation}
				\begin{aligned}
					\mathfrak r_{lk} &= \mathfrak p_{k} \mathfrak r_{(l-1)k},
					\\
					\mathfrak r_{lk+1} &=\mathfrak p_{k} \mathfrak r_{(l-1)k+1} +
					\mathfrak p_{k+1} \mathfrak r_{(l-1)k} -i \partial _{a} \mathfrak p_{k}
					\delta _{a}\mathfrak r_{(l-1)k},
					\\
					\mathfrak r_{lk+2} &=\mathfrak p_{k} \mathfrak r_{(l-1)k+2} +
					\mathfrak p_{k+1} \mathfrak r_{(l-1)k+1} +\mathfrak p_{k+2}
					\mathfrak r_{(l-1)k}
					\\
					& \quad -i\partial _{a}\mathfrak p_{k}\delta _{a}\mathfrak r_{(l-1)k+1}
					-i\partial _{a}\mathfrak p_{k+1} \delta _{a} \mathfrak r_{{l-1}k} -
					\frac{1}{2}\partial _{a}\partial _{b}\mathfrak p_{k}\delta _{a}
					\delta _{b}\mathfrak r_{(l-1)k}.
				\end{aligned}
				\label{ppl}
			\end{equation}
			The formula for $\mathfrak r_{lk}$ is obvious. Consider now the expression
			for $\mathfrak r_{lk+1}$ and assume that the formula holds for $l-1$, then
			\begin{align*}
				\mathfrak r_{lk+1} & = \mathfrak p_{k} \left ( \sum \limits _{j = 1}^{l-1}
				(\mathfrak p_{k})^{j-1} \mathfrak p_{k+1} (\mathfrak p_{k})^{l-1-j} - i
				\sum \limits _{1 \leq j < s \leq l-1} {\partial}_{a}^{[j]} {\delta}_{a}^{[s]}
				(\mathfrak p_{k})^{l-1} \right )
				\\
				& \qquad + \mathfrak p_{k+1} \mathfrak p_{k}^{l-1} -i \partial _{a}
				\mathfrak p_{k} \delta _{a}\mathfrak p_{k}^{l-1}
				\\
				& = \left ( \mathfrak p_{k} \sum \limits _{j = 1}^{l-1} (\mathfrak p_{k})^{j-1}
				\mathfrak p_{k+1} (\mathfrak p_{k})^{l-1-j} + \mathfrak p_{k+1}
				\mathfrak p_{k}^{l-1} \right )
				\\
				& \quad - i \left ( \mathfrak p_{k} \sum \limits _{1 \leq j < s \leq l-1}
				{\partial}_{a}^{[j]} {\delta}_{a}^{[s]} (\mathfrak p_{k})^{l-1} +
				\partial _{a} \mathfrak p_{k} \delta _{a}\mathfrak p_{k}^{l-1}
				\right )
				\\
				& = \sum \limits _{j = 1}^{l} (\mathfrak p_{k})^{j-1} \mathfrak p_{k+1}
				(\mathfrak p_{k})^{l-j} -i \sum \limits _{1 \leq j < s \leq l} {
					\partial}_{a}^{[j]} {\delta}_{a}^{[s]} (\mathfrak p_{k})^{l},
			\end{align*}
			where we have used only the reordering of terms and the Leibniz rule for
			$\delta _{a}$:
			\begin{equation*}
				\delta _{a} (\mathfrak p_{k})^{l-1} = \sum \limits _{s=1}^{l-1}
				\delta _{a}^{[s]} (\mathfrak p_{k}) l-1.
			\end{equation*}
			
			The formula for $\mathfrak r_{lk+2}$ can be proved in a similar
			way. We can split the product in \reftext{\eqref{ppl}} into the sum with no derivatives,
			only first-order derivatives, and finally second-order derivatives. We
			skip the proof for the first two parts and illustrate only the last part.
			The component of $\mathfrak r_{lk+2}$ with second-order derivatives (denoted
			$\mathfrak r_{lk+2}(2)$) will come from the following contributions.
			%
			\begin{equation}
				\begin{aligned}
					\mathfrak r_{lk+2}(2) &=\mathfrak p_{k} \mathfrak r_{(l-1)k+2}(2) -i
					\partial _{a}\mathfrak p_{k}\delta _{a}\mathfrak r_{(l-1)k+1}(1) -
					\frac{1}{2}\partial _{a}\partial _{b}\mathfrak p_{k}\delta _{a}
					\delta _{b}\mathfrak r_{(l-1)k} = \cdots
				\end{aligned}
				\label{eqA.9}
			\end{equation}
			which, after assuming the validity for $l-1$ gives,
			%
			\begin{equation}
				\begin{aligned}
					\qquad \qquad \cdots &=\mathfrak p_{k} \left ( -\frac{1}{2}\sum
					\limits _{1 \leq j < p \leq l} \,\, \sum \limits _{1 \leq r < s \leq n}
					{\partial}_{b}^{[r]} {\delta}_{b}^{[s]} {\partial}_{a}^{[j]} {\delta}_{a}^{[p]}
					(\mathfrak p_{k})^{l} \right )
					\\
					& - i \partial _{b}\mathfrak p_{k} \delta _{b} \left ( -i \sum
					\limits _{1 \leq j < w \leq l} {\partial}_{a}^{[j]} {\delta}_{a}^{[w]}
					(\mathfrak p_{k})^{l} \right )
					\\
					& - \frac{1}{2} \partial _{a}\partial _{b}\mathfrak p_{k}\delta _{a}
					\delta _{b} \mathfrak p_{k}^{l-1}.
				\end{aligned}
				\label{eqA.10}
			\end{equation}
			Again, using the Leibniz rule, we see that it is indeed the term from \reftext{\eqref{operator_powers}} with second-order derivatives split into the parts
			that first $\mathfrak p_{k}$ has no derivatives acted upon, one derivative,
			and two derivatives.
		\end{proof}
		%
		\begin{coro}
			\label{corA.2}
			In the special case of scalar symbols (a sufficient condition is that symbol
			$\mathfrak p_{k}$ is scalar) we have (with the same notation as above),
			%
			%
			\begin{equation}
				\begin{aligned}
					\mathfrak r_{lk} &= (\mathfrak p_{k})^{l}, \\
					\mathfrak r_{lk+1} &= l (\mathfrak p_{k})^{l-1}\mathfrak p_{k+1} -i
					\frac{l(l-1)}{2}(\mathfrak p_{k})^{l-2} \partial _{a}(\mathfrak p_{k})
					\delta _{a}(\mathfrak p_{k}), \\
					\mathfrak r_{lk+2} &= l (\mathfrak p_{k})^{l-1}\mathfrak p_{k+2} +
					\frac{l(l-1)}{2}(\mathfrak p_{k})^{l-2}(\mathfrak p_{k+1})^{2}
					\\
					& \qquad -i\frac{l(l-1)}{2}(\mathfrak p_{k})^{l-3} \biggl[
					\mathfrak p_{k} \Big(\partial _{a}(\mathfrak p_{k+1})\delta _{a}(
					\mathfrak p_{k}) +\partial _{a}(\mathfrak p_{k})\delta _{a}(
					\mathfrak p_{k+1})\Big)
					\\
					& \qquad \qquad \qquad +(l-2)\mathfrak p_{k+1}\partial _{a}(
					\mathfrak p_{k}) \delta _{a}(\mathfrak p_{k})\biggr]
					\\
					& \qquad -\frac{l(l-1)}{24}(\mathfrak p_{k})^{l-4} \biggl( 6(
					\mathfrak p_{k})^{2}\partial _{a}\partial _{b}(\mathfrak p_{k})
					\delta _{a}\delta _{b}(\mathfrak p_{k})
					\\
					& \qquad \qquad \qquad + 3(l-2)(l-3)\partial _{a}(\mathfrak p_{k})
					\partial _{b}(\mathfrak p_{k}) \delta _{a}(\mathfrak p_{k})\delta _{b}(
					\mathfrak p_{k})
					\\
					&\qquad \qquad \qquad +4(l-2)\mathfrak p_{k} \biggl[ \partial _{a}(
					\mathfrak p_{k})\partial _{b}(\mathfrak p_{k}) \delta _{a}\delta _{b}(
					\mathfrak p_{k}) \\
					&  \qquad \qquad \qquad
					+ \partial _{a}(\mathfrak p_{k})\partial _{b}\delta _{a}(
					\mathfrak p_{k})\delta _{b}(\mathfrak p_{k})
					+ \partial _{a}\partial _{b}(\mathfrak p_{k})
					\delta _{a}(\mathfrak p_{k}) \delta _{b}(\mathfrak p_{k}) \biggr]
					\biggr).\label{commutative_operator_powers}
				\end{aligned}
			\end{equation}
		\end{coro}

\end{document}